\documentclass[twoside]{amsart}
\usepackage{latexsym}
\usepackage{amssymb,amsmath,amsopn}
\usepackage[dvips]{graphicx}   
\usepackage{color,epsfig}      

               {\begin{list}{}{\leftmargin#1\rightmargin#2}\item{}}%
               {\end{list}}
 
\newtheorem{thm}{Theorem}
\newtheorem{lem}{Lemma}

\theoremstyle{definition}
\newtheorem{defn}{Definition}
\newtheorem{rem}{Remark}

\renewcommand{\Re}{\mathbb R}

\def\bea{\begin{eqnarray}}
\def\eea{\end{eqnarray}}

\DeclareMathOperator{\grad}{grad}

\DeclareMathOperator{\inter}{int}

\DeclareMathOperator{\conv}{conv}

\begin{document}
\title[Tracking equilibrium points]{Tracking critical points on evolving curves and surfaces}
\author[G. Domokos, Z. L\'angi \and A. \'A. Sipos]{G\'abor Domokos, Zsolt L\'angi \and Andr\'as \'Arp\'ad Sipos}

\address{G\'abor Domokos, MTA-BME Morphodynamics Research Group and Dept. of Mechanics, Materials and Structures, Budapest University of Technology,
M\H uegyetem rakpart 1-3., Budapest, Hungary, 1111}
\email{domokos@iit.bme.hu}
\address{Zsolt L\'angi, MTA-BME Morphodynamics Research Group and Dept. of Geometry, Budapest University of Technology,
Egry J\'ozsef u. 1., Budapest, Hungary, 1111}
\email{zlangi@math.bme.hu}
\address{Andr\'as \'A. Sipos, MTA-BME Morphodynamics Research Group and Dept. of Mechanics, Materials and Structures, Budapest University of Technology,
M\H uegyetem rakpart 1-3., Budapest, Hungary, 1111}
\email{siposa@eik.bme.hu}


\begin{abstract}

In recent years it became apparent that geophysical abrasion can be well characterized by the time evolution $N(t)$ of the number $N$ of static balance points of the abrading particle. Static balance points correspond to the critical points of the particle's surface represented as a scalar distance function $r$, measured from the center of mass of the particle, so their time evolution can be expressed as $N(r(t))$. The mathematical model of the particle can be constructed on two scales: on the macro (global) scale the particle may be viewed as a smooth,  convex manifold described by the smooth distance function $r$ with $N=N(r)$ equilibria, while on the micro (local) scale the particle's natural model is a finely discretized, convex polyhedral approximation  $r^{\Delta}$ of $r$, with $N^{\Delta}=N(r^{\Delta})$ equilibria.  There is strong intuitive evidence suggesting that under some particular evolution models (e.g. curvature-driven flows) $N(t)$ and $N^{\Delta}(t)$ primarily evolve in the opposite manner (i.e. if one is increasing then the other is decreasing and vice versa). This observation appear to be a key factor 
in tracking geophysical abrasion. Here we  create the mathematical framework necessary to understand these phenomenon more broadly, regardless of the particular evolution equation. We study micro and macro events in one-parameter families of curves and surfaces, corresponding to bifurcations triggering the jumps in $N(t)$ and $N^{\Delta}(t)$. Based on this analysis we show that the intuitive picture developed for curvature-driven flows is not only correct, it has universal validity,   as long as  the evolving surface $r$ is smooth. In this case, bifurcations associated with $r$ and $r^{\Delta}$ are coupled to some extent: resonance-like phenomena in $N^{\Delta}(t)$ can be used to forecast downward jumps in $N(t)$ (but not upward jumps). Beyond proving rigorous results for the $\Delta \to 0$ limit on the nontrivial interplay between singularities in the discrete and continuum approximations we also show that our mathematical model is structurally stable, i.e. it may be verified by computer simulations.
\end{abstract}
\maketitle

\tableofcontents

\section{Introduction}

\subsection{Motivation}

Recent work in geomorphology \cite{Domokos, Litwin, Williams} indicates that the shapes of sedimentary particles and the time ($t$) evolution of those shapes may be well characterized by the number $N(t)$ of mechanical balance points of the abrading particle. Such balance points correspond to the critical points of the scalar distance function $r$, measured from the center of mass $o$.  In this paper we develop the mathematical theory for the case when $r$ is a smooth, convex planar curve but we will also show numerical results for the non-smooth case and for surfaces.

There are two types of physical models describing the evolution of particles under abrasion. The first kind, which we might call  \emph{local model}, is based on discrete events when in a collision a small part of the abraded particle is broken off. The most natural geometrical setting for local models is a multi-faceted convex polyhedron and collisional events correspond to truncations with planes parallel to, and very close to tangent planes.  The second kind of model we might call  \emph{global} and it considers the averaged effect of many such micro collisions. The natural setting for global models is a smooth, convex body evolving under a geometric partial differential equation (PDE). Both model types are physically legitimate: at close inspection, the convex hulls of pebbles can be best approximated by multi-faceted polyhedra, on the other hand, it is equally possible to adopt the global view and approximate pebbles with smooth surfaces.

\begin{figure}
\includegraphics[width=\textwidth]{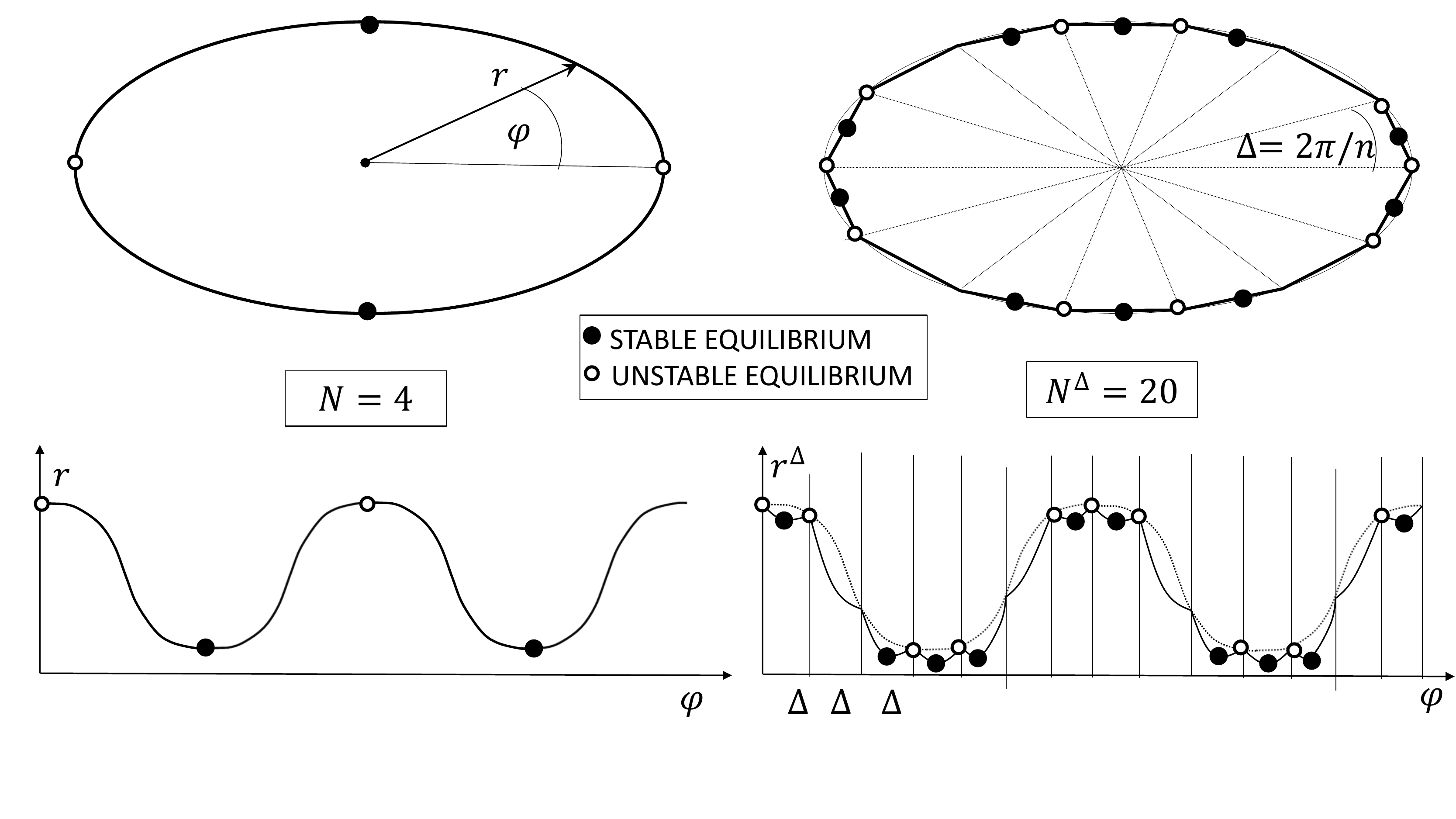}
\caption[]{Equilibria on a planar ellipse and its polygonal approximation. Upper left: ellipse in Euclidean plane. Lower left: distance function $r(\varphi)$ corresponding to the ellipse. Upper right: polygonal approximation in Euclidean plane. Lower right: distance function  $r^{\Delta}(\varphi)$ corresponding to polygon. Observe that most, but not all edges and vertices of the polygon carry equilibria. The latter accumulate in flocks, centered around the locations of the equilibria of the smooth curve.}
\label{fig:planar}
\end{figure}
The distance functions describing these two types of models are, of course, related: in case of a fine polyhedral approximation the two surfaces (smooth global model and polyhedral discrete model) are close to each other in the $C^1$-norm. We denote the size of the largest polyhedral face by $\Delta$ and the two distance functions by $r$ and $r^{\Delta}$, respectively. Mechanical equilibria correspond to the critical points of $r$ and $r^{\Delta}$. We will refer to these points as global and local equilibria and denote their numbers by $N$ and $N^{\Delta},$ respectively.   Both the smooth function $r$ and its polyhedral approximation $r^{\Delta}$ may carry equilibria of different stability types. In three dimensions we have three generic types: stable, saddle an unstable. For example, vertices of a polyhedron may carry unstable equilibrium points, edges may carry saddle-type equilibrium points and faces may carry stable equilibrium points. The numbers $N$ and $N^{\Delta}$ refer to the number of equilibria belonging to any of the aforementioned stability types. In two dimensions we just have two generic stability types: stable and unstable equilibria follow each other alternating along the smooth curve $r$ or its fine polygonal approximation $r^{\Delta}$. Figure \ref{fig:planar} illustrates these concepts for a planar, elliptical disc and its discretized, polygonal approximation. Although critical points appear to be related to first derivatives (they are defined by vanishing gradient), nevertheless, the $C^1$-proximity of the two functions does not imply that $N$ and $N^{\Delta}$ are close. As we showed in \cite{Monatshefte}, in the $\Delta \to 0$ limit $N^{\Delta}$ does not, in general, converge to $N$.

The time dependence $N(t)$ of the number of global critical points has been broadly investigated in various evolution equations   \cite{Damon, Domokos, Matgeo, Eikonal, Grayson, Kuijper}. Our goal here is rather different: instead of studying any particular evolution equation (which we will use only as illustrations) we focus on some universal features relating $N(t)$ to $N^{\Delta}(t)$.  Earlier results appear to suggest intuitively that,  as long as $r(t)$ is smooth, $N(t)$ and $N^{\Delta}(t)$ tend to evolve in the opposite directions: in \cite{VarkonyiDomokos} it was shown that it is always possible to increase the number of equilibria via suitable, small truncations, however, the opposite is not true: in general, it is not possible to reduce the number of equilibria by a small local truncation. These results were further advanced in \cite{robust} where the concept of robustness
was introduced to measure the stability of $N$ with respect to truncations of the solid. By distinguishing between upward and downward robustness (measuring the difficulty to increase or decrease $N$, respectively) it was again found that upward robustness is, in general, much smaller than downward robustness. These results suggest that in the local, polygonal model, which could be realized in a randomized chipping algorithm \cite{Sipos,Krapivsky}, $N^{\Delta}(t)$ would tend to increase under subsequent, small random truncations. On the other hand, there are results \cite{Grayson, Domokos} showing that, at least in curvature-driven, global PDE models which could be regarded as continuum analogies of the aforementioned chipping algorithms, $N(t)$ tends to decrease.

In this paper we will show that the indicated opposite trend of  $N^{\Delta}(t)$ and  $N(t)$ is indeed universal, and it is independent of the particular type of evolution model
as long as $r(t)$ remains smooth. Our paper will focus on this case, and we will show that resonance-like phenomena in $N^{\Delta}(t)$ may help to predict \emph{downward} jumps in $N(t)$.  
  The nontrivial coupling between  $N^{\Delta}(t)$ and  $N(t)$ may be better understood  intuitively via an analogy to a mechanical oscillator.
In case of of a damped, driven harmonic oscillator resonance occurs whenever the driving frequency approaches the natural frequency of the oscillator, i.e. an \emph{extrinsic} quantity approaches and \emph{intrinsic} one. In our problem
we associate two scalars with a point $p$ of a smooth curve: the distance $r(p)$ between $p$ and the center of mass $o$,
and the radius of curvature $R(p)$ at $p$. In the analogy, $r(p)$ is the extrinsic and $R(p)$ is the intrinsic quantity. (In three dimensions we have two intrinsic quantities: the two principal radii) The size $\Delta$ of the discretization is analogous to damping  and $N^{\Delta}$ is analogous to the amplitude of the oscillation. If $r(p)=R(p)$ then in the $\Delta \to 0$ limit we can observe as $N^{\Delta}\to \infty$.  
However,  the analogy is incomplete, because in the case of the harmonic oscillator 
a single amplitude-frequency diagram is sufficient to describe the generic response of the system. In our geometric setting
there exist two, distinct generic scenarios which we explain below.

As $r(p) \to R(p)$,  the trajectory of the  center of mass $o$ approaches the evolute $E^{\Delta}$ or $E$ corresponding to $r^{\Delta}$ and $r$, respectively \cite{robust, Poston}, see also Remarks \ref{rm:caustic1}, \ref{rm:caustic2}. We will refer to the intersection between the trajectory of $o$ and the evolute $E^{\Delta}$ or $E$  as micro and macro events, respectively. It is well-known \cite{Poston} that these intersections trigger upward or downward jumps in the integer-valued functions $N^{\Delta}(t)$ and  $N(t)$. In the generic case an event is equivalent to a codimension one, saddle-node bifurcation. Such a bifurcation can
occur in two different manners: either creating or annihilating one pair of equilibrium points. As we will show, micro events, in general, do not trigger macro events, however, the opposite is not true: jumps in $N(t)$ occur whenever $r(p)=R(p)$ and 
here we have always $\lim_{\Delta \to 0}N^{\Delta}=\infty$. However, the time evolution $N^{\Delta}(t)$ will depend on thy type
of macro-event in an asymmetric manner (cf. Figure \ref{fig:events}):
\begin{itemize}
\item
If $N(t)$ increases by 2 then we call this a (generic) creation
and the corresponding macro event in $N^{\Delta}(t)$ will be of type $C$.
\item If $N(t)$ decreases by 2 then we call it a (generic) annihilation and the corresponding macro event in
$N^{\Delta}(t)$ will be of type $A$.
\end{itemize}
We will prove the existence of these macro events in Theorem \ref{thm:caustic} and discuss the exact evolution
of $N^{\Delta}$ in their vicinity in Section \ref{ss:jump}.

\begin{figure}
\includegraphics[width=\textwidth]{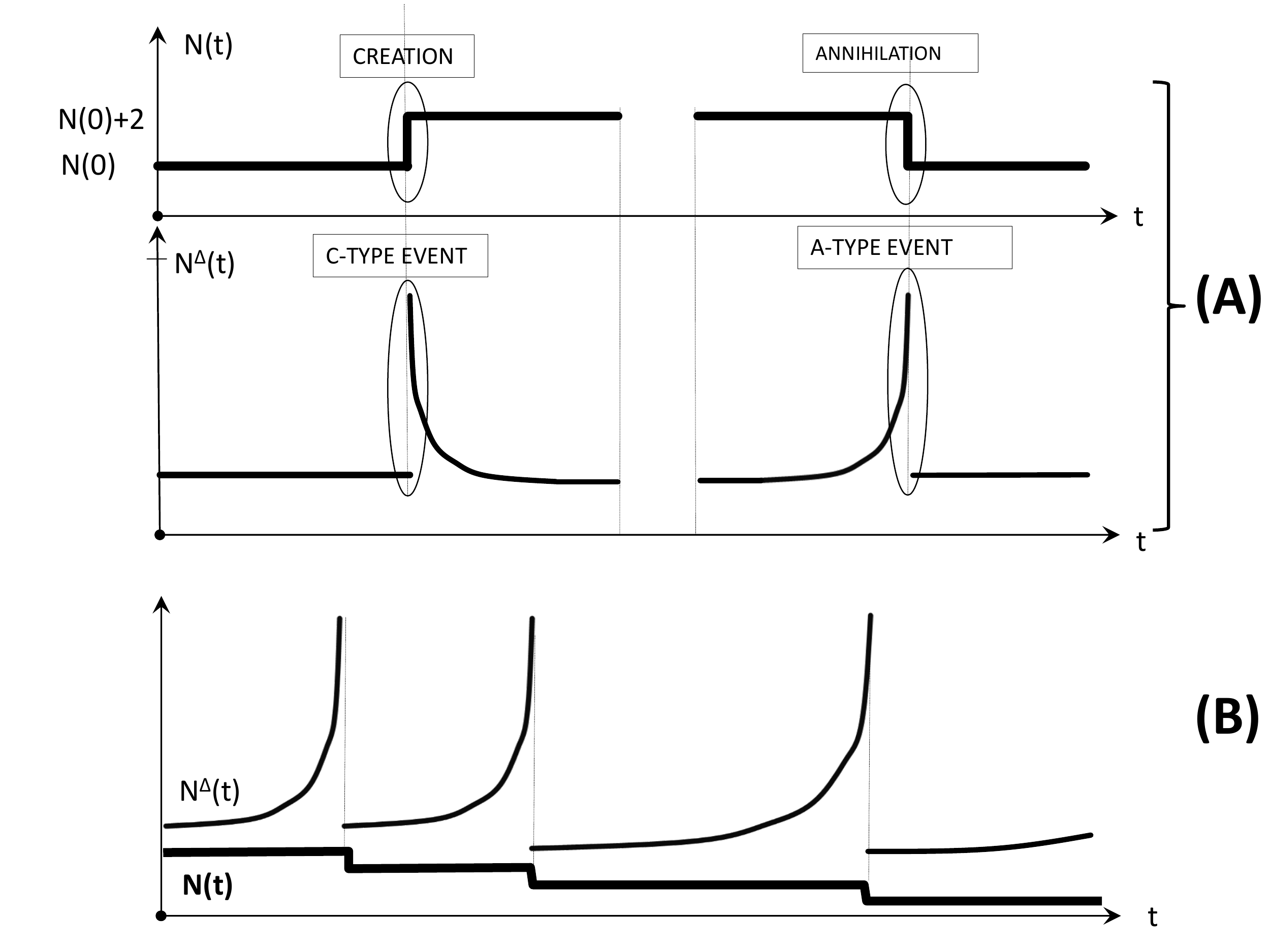}
\caption[]{Co-evolution of $N(t)$ and $N^{\Delta}(t)$. (A) Related macro events
(B) Qualitative example explaining the observations of "opposite" evolution of $N(t)$ and $N^{\Delta}(t)$.
Observe that, although the value $N^{\Delta}(t)$ may
not increase in an averaged sense, at almost all times $N^{\Delta}(t)$ appears to be increasing, while
$N(t)$ is monotonically decreasing.}
\label{fig:events}
\end{figure}

The existence of coupled macro-events sheds light on the previous intuition about the opposite co-evolution of $N^{\Delta}(t)$ and  $N(t)$. As we can see in Figure \ref{fig:events}(B), if $N(t)$ evolves in a monotonic fashion (i.e. it has jumps only in one direction) then, due to the coupling between corresponding macro events,  for \emph{most of the time} the evolution of $N^{\Delta}(t)$ will be in the opposite direction. This does not imply that $N^{\Delta}(t)$ will be monotonic in any averaged sense,
however, in a generic case, locally it will almost always appear to be monotonic. Figure \ref{fig:events}(B) illustrates the qualitative trends
for $N^{\Delta}(t)$  and $N(t)$ in the $\Delta \to 0$ limit.  Since the mesh-size $\Delta$ is analogous to damping, in
computer simulations, for finite values of $\Delta$ we expect to see  finite versions of $C$-type
and $A$-type events as well as  small fluctuations of $N^{\Delta}(t)$, cf. Figure \ref{fig:coev}.  

Beyond explaining earlier observations, these results can also be of practical use. Both computer simulations of the PDEs describing abrasion processes and related laboratory measurements are inherently discrete, one good example is the study of surfaces of natural pebbles  which, while rolling on a horizontal plane, are supported on their convex hull \cite{Domokosetal}. The latter is well approximated by a many-faceted polyhedron $r^{\Delta}$ (with faces of maximal size $\Delta$), on which, by studying its detailed 3D scanned images,
one can clearly observe large numbers $N^{\Delta}$ of adjacent equilibria in strongly localized \emph{flocks}. The macro events of type $A$ and $C$ in the evolution of $N^{\Delta}(t)$ correspond to the explosion of these local flocks into huge \emph{critical flocks} the size and evolution of which we explore in Subsections \ref{ss:size} and \ref{ss:time}.
If we approximate the polyhedron $r^{\Delta}$  by a sufficiently smooth surface $r$, then, in a generic case, we can see that the flocks of equilibria on $r^{\Delta}$ appear in the close vicinity of the (isolated) equilibrium points of $r$ (cf. Figure \ref{fig:pebble}), however, the latter may not be directly observed on the polyhedral image. In computer simulations the opposite happens: a smooth surface $r$ is replaced by its fine $r^{\Delta}$ discretization and computations are performed on the latter. As we can see, it is often the case that we have means to monitor $N^{\Delta}(t)$ but we may not be able to directly monitor $N(t)$, although the latter is of prime physical interest \cite{Domokos}. In such cases by using Theorem \ref{thm:caustic}, $N(t)$ may be obtained simply via monitoring the $C$- and $A$-type events in $N^{\Delta}$ and computing $N(t)$ as
\begin{equation}
N(t)=N(0)+2(c-a),
\end{equation}
where $c$ and $a$ refer respectively to the number of $C$-type and $A$-type events observed in $N^{\Delta}(t)$.

These macro events also connect $N^{\Delta}$ to the aforementioned robustness concepts. The quantity $\sigma(t)=1/N^{\Delta}(t)$ may serve as a measure of downward robustness since whenever  $\sigma(t) \to 0$, the function $N(t)$ is approaching a downward step. Curiously,  $\sigma(t)$ does not carry  advance information on the approaching upward step in $N(t)$.
 The asymmetry in increasing/decreasing $N(t)$ is at the very heart of understanding natural abrasion processes and their mathematical models. One, rather delicate feature of these PDEs is whether they tend to increase, conserve or reduce $N$ \cite{Grayson, Domokos} and the ability to track and measure this phenomenon in experiments and computations is of key importance in the identification and scaling of the proper evolution equations. The dynamic theory for local equilibria, the central topic of this paper, appears to be a necessary step towards this goal.
Our paper is structured as follows.

To understand their evolution, the first natural question is to ask for the relationship between  $N^{\Delta}$ and $N$ on a fixed surface described by a generic distance function $r$. We addressed this problem (which we may call the \emph{static theory of equilibria}) in \cite{Monatshefte} and we obtained explicit formulae (to be reviewed in Subsection \ref{int:monatshefte}) for the size of the individual flocks (emerging in the $\Delta \to 0$ limit) surrounding generic critical points of $r$. However, those results do \emph{not}  permit the computation of the global value $N^{\Delta}$ on the whole surface because in \cite{Monatshefte} we did not exclude the existence of equilibria outside flocks. In Subsection \ref{ss:irregular} we complement the static theory by filling this gap; we will prove that local equilibria disconnected from flocks do not exist. In Subsection \ref{ss:random} we further strengthen these results by showing the structural stability of our formulae with respect to small random fluctuations in mesh size. The latter result validates computer simulations running with slightly unequal mesh size. The main focus of our current paper is the \emph{dynamic theory} which we develop in Section \ref{ss:jump}. We prove our results for planar curves, however, in Section \ref{sec:3D} we provide a visually attractive numerical example for the evolution of critical flocks on surfaces.

\begin{figure}
\includegraphics[width=\textwidth]{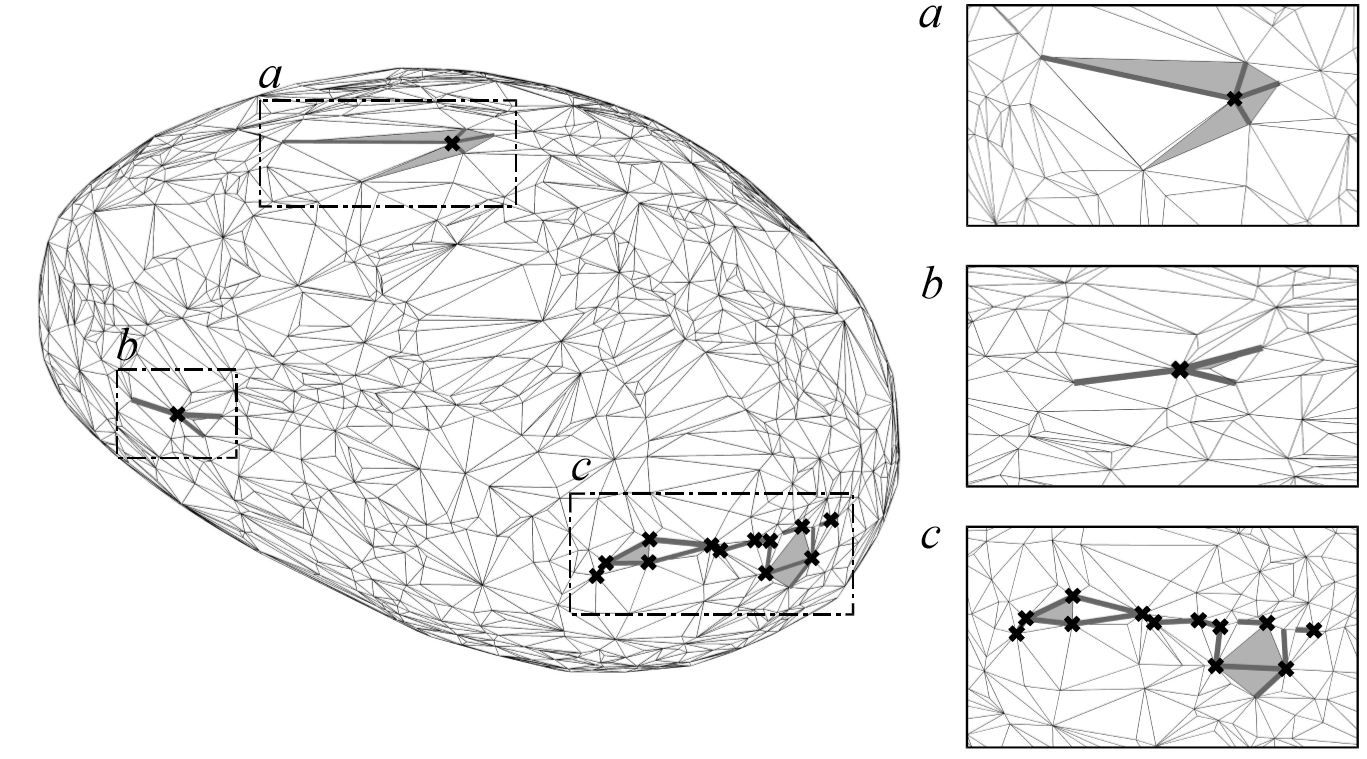}
\caption[]{ Flocks of equilibria observed on a scanned pebble. Polyhedron model of pebble displaying multiple equilibria
of all three stability types: vertices carrying (unstable) equilibria are marked with "X", edges carrying (saddle-type) equilibria marked with thick black line,  faces carrying (stable) equilibria are shaded grey. Unmarked vertices, edges and faces do not carry equilibria. Note that the marked equilibria
are spatially strongly localized, concentrated in areas marked by (a), (b) and (c). Such accumulation of polyhedral equilibria are referred to as "`flocks"'. Flocks appear in the vicinity of the equilibria of the smooth surface approximated by the multi-faceted polyhedron. Similar phenomena can be observed by finely discretized smooth curves where equilibria of the polygon form flocks around equilibria of the smooth curve.}
\label{fig:pebble}
\end{figure}

\subsection{Basic concepts}

\subsubsection{Mechanical equilibria}\label{int:mech}

The study of  equilibria of rigid bodies was initiated by Archimedes \cite{Archimedes1}; his results have been used in naval design even in the 18th century (cf. \cite{Archimedes2}). Archimedes' main concern was the number of the stable balance points of the body.

Mechanical equilibria of convex bodies correspond to the singularities of the gradient vector field characterizing their surface. Modern, global theory of the generic singularities of smooth manifolds appears to start with
the papers of Cayley \cite{Cayley} and Maxwell \cite{Maxwell} who independently introduced topological ideas into this field yielding results on the global number of stationary points. These ideas have been further generalized
by Poincar\'e and Hopf, leading to the Poincar\'e-Hopf Theorem \cite{Arnold} on topological invariants.
In case of topological spheres in two and three dimensions, the Poincar\'e-Hopf formula can be written as
\begin{equation} \label{Poincare}
\mbox{2D: } S-U=0, \mbox{    3D: } S+U-H=2,
\end{equation}
where $S,U,H$ denote the numbers of `sinks' (minima, corresponding to stable equilibria), `sources' (maxima, corresponding to unstable equilibria) and saddles, respectively. This formula can be also regarded as a variant
of the well-known Euler's formula \cite{Euler} for convex polyhedra.

Mechanical equilibria of polyhedra have also been investigated in their own right; in particular, the minimal number of equilibria attracted substantial interest. Monostatic polyhedra (i.e. polyhedra with just $S=1$ stable equilibrium point) have been studied in \cite{Conway},\cite{Dawson}, \cite{DawsonFinbow} and \cite{Heppes}.

The total number $N$ of equilibria $(N= S+U+H)$ has also been in the focus of research. In planar, homogeneous, convex bodies (rolling along their circumference on a horizontal support), we have $N\geq 4$ \cite{Domokos1}. However, convex homogeneous objects with $N=2$ exist in the three-dimensional space (cf. \cite{VarkonyiDomokos}).  Zamfirescu \cite{Zamfirescu} showed that for \emph{typical} convex bodies, $N$ is infinite.

\subsubsection{Fine discretizations: earlier results}\label{int:monatshefte}

While typical convex bodies are neither smooth objects, nor are they polyhedral surfaces, Zamfirescu's result
strongly suggests that equilibria in abundant numbers may occur in physically relevant scenarios. This is indeed the case if we study the surfaces of natural pebbles  which, while rolling on a horizontal plane, are supported on their convex hull \cite{Domokosetal} and exhibit \em flocks \rm of equilibria (cf. Figure \ref{fig:pebble}).

In \cite{Monatshefte} we provided a mathematical justification for this observation. We studied the inverse phenomenon: namely, we were seeking the numbers and types of static equilibrium points of the families of polyhedra $r^{\Delta}$ arising as equidistant $\Delta$-discretizations on an increasingly refined grid  of a smooth curve $r$ with $N$ generic equilibrium points, denoted by $m_i$ $(i=1,2,\dots N$). In the planar case, as $\Delta \to 0$, $r^{\Delta} \to r$ and we find that the diameter of each of the $N$ flocks on $r^{\Delta}$ (appearing around $m_i$) shrink and approach zero.
However, we also find that inside a fixed domain (centered at $m_i$), the numbers $S_i^{\Delta},U_i^{\Delta}$ of equilibria in each flock fluctuate around specific values  $S_i^0$ and $U_i^0$  that are independent of the mesh size
and the parametrization of the surface. We called these quantities the \emph{imaginary equilibrium indices} associated with $m_i$. We may eliminate the fluctuation of  $S_i^0$, $U_i^0$ by averaging over meshes in random
positions (with uniform distributions) and in Theorem 1 of \cite{Monatshefte} we obtained for the planar case that
\begin{equation}\label{mainresult2D}
S_i^0=1/|(\kappa_i\rho_i+1)|,  \quad \\ U_i^0= |\kappa_i\rho_i|/|(\kappa_i\rho_i+1)|,
\end{equation}
where $\rho_i = |m_i|$, and $\kappa_i$ is the (signed) curvature of $r$ at $m_i$. What we did \emph{not} prove was whether by summing over all imaginary equilibrium indices 
\begin{equation} \label{local2D}
N^0=\sum_{i=1}^{N} (S_i^0+U_i^0)
\end{equation}
actually provides \emph{all} equilibria on the $\Delta \to 0$ mesh. We will provide this result for the  $2$-dimensional case in Subsection \ref{ss:irregular} by proving Theorem \ref{thm:no_irregular} claiming that there exist no `irregular' equilibria on the $\Delta \to 0$ mesh which are separated from the flocks. In Subsection \ref{ss:random} we will prove Theorem \ref{thm:stochastic} about randomized meshes indicating that these formulae are robust and they approximately predict results computed on non-uniform meshes.

\section{Static theory: equilibria on finely discretized, fixed planar curves} \label{s:curves}

Throughout this section, we deal with a planar curve satisfying the $C^3$ differentiability property which has exactly one, non-degenerate equilibrium point $m$ with respect a given reference point $o$. Note that as a plane curve is a one-dimensional submanifold of $\Re^2$, there is a neighborhood of $m$ in which the examined curve can be defined as a simple $r : [\underline{\tau},\overline{\tau}] \to \Re^2$, $r(\tau):=(x(\tau),y(\tau))$ three times continuously differentiable curve, where $\underline{\tau} < 0 < \overline{\tau}$ and $m=r(0)$.
For the evolution of static equilibria we may write in  a more explicit notation $N(r(t,\tau),o(t))$ where $\tau$ is the spatial parametrization of $r$ and for planar curves $\tau$ is a scalar, for surfaces it is a vector. We will use the shorthand notation $N(t)$ if it is clear from the context which function $r$ and which reference point $o$ are involved.

By a suitable choice of the coordinate system, we may assume that our reference point $o$ is the origin, and that $m=r(0) = (0,\rho)$ is on the positive half of the $y$-axis; i.e. $\rho > 0$. This implies that $\dot{y}(0) = 0$, and we also assume that $\dot{x} > 0$.
We restrict our investigation to curves that are `locally convex' with respect to $o$, and whose equilibrium point is non-degenerate.
In other words, we assume that the signed curvature $\kappa$ of $r(\tau)$ at $m$ satisfies the inequalities $0 \neq 1+\kappa \rho < 1$.

Let $F^{\Delta}$ denote the $n$-segment equidistant partition of $[\underline{\tau},\overline{\tau}]$ with $\Delta = \frac{(\overline{\tau}-\underline{\tau})}{n}$.
If $i\Delta$ is a division point of $F^{\Delta}$, then we introduce the notation
$p^{\Delta}_i = r (i\Delta)$, where $\underline{\tau} \leq i\Delta \leq \overline{\tau}$.
Note that the indices of the points $p_i^{\Delta}$ are not necessarily integers, but real numbers that are congruent $\mod 1$. We denote the set of indices of the points $p_i^{\Delta}$ by $I$, and examine the equilibria of the approximating curve $P^{\Delta} = \bigcup_{i-1,i \in I} [p_{i-1}^{\Delta} , p_i^{\Delta}]$.
During our investigation, we assume that these equilibria are generic; that is, if $P^{\Delta}$ has an equilibrium point at a vertex $p_i^{\Delta}$, then the vector $p_i^{\Delta}$ is perpendicular to neither $p_{i+1}^{\Delta}-p_i^{\Delta}$ nor $p_{i-1}^{\Delta}-p_i^{\Delta}$.

For any $K > 0$, let $S^{\Delta}(K)$ denote the number of stable equilibria of $P^{\Delta}$ with respect to $o$ lying on the sides $[p_i^{\Delta},p_{i+1}^{\Delta}]$ satisfying $|i| \leq K$. We define the quantity $U^{\Delta}(K)$ for unstable equilibria of $P^{\Delta}$ analogously. In Theorem 1 of \cite{Monatshefte}, we proved that if $K=K(\rho,\kappa)$ is sufficiently large, then  in the ${\Delta} \to 0$ limit $S^{\Delta}$ and $U^{\Delta}$ will fluctuate around the so-called imaginary equilibrium indices given in equation (\ref{mainresult2D}).

Roughly speaking, we may say that these formulas hold for number of equilibria with `bounded' indices, and we observe that this result holds for \emph{all} sufficiently small $\Delta$, and the quantities in the formulas are independent of both $\Delta$ and the parametrization of the curve.

\subsection{Nonexistence of irregular equilibria} \label{ss:irregular}

Whereas it is easy to see that the equilibria of $P_{\Delta}$ are `physically' close to $m$ (i.e. for any equilibrium point $q$, $|q-m|$ is arbitrarily close to $0$ if $\Delta$ is sufficiently small), this property does not imply that the set of indices of equilibrium points is bounded by some $K$ independent of $\Delta$. Thus, the formulas in (\ref{mainresult2D}) do not necessarily provide the numbers of \emph{all} equilibrium points of $P^{\Delta}$.
This led in \cite{Monatshefte} to the following definition:

\begin{defn}
If $\{ p_{i_k}^{\Delta_k} \}$ is a sequence of equilibrium points of $P^{\Delta_k}$ with
$\lim\limits_{k \to \infty} \Delta_k = 0$ and $\lim\limits_{k \to \infty} i_k = \infty$, then the sequence $\{ p_{i_k}^{\Delta_k} \}$ is called an \emph{irregular} equilibrium sequence.
\end{defn}

In \cite{Monatshefte} it is remarked that if the coordinate functions of $r(\tau)$ are polynomials, then the curve has no irregular equilibrium sequence, and asked (Question 1, \cite{Monatshefte}) whether the same holds if the two coordinate functions are analytic.
Here we prove that the answer to this question is affirmative not only for every analytic, but for every $C^3$-class curve. This means that the quantities in (\ref{mainresult2D}) are valid for \emph{all} equilibrium points of $P^{\Delta}$, if ${\Delta}$ is sufficiently small.

\begin{thm}\label{thm:no_irregular}
If $\Delta$ is sufficiently small then there exist some values $k(\Delta),K(\Delta)$ such that $[p_i^{\Delta}, p_{i+1}^{\Delta}]$ contains an equilibrium point if, and only if $k(\Delta) \leq i \leq K(\Delta)$; i.e. for sufficiently small ${\Delta}$ the set of indices of equilibrium points is `connected'. In particular, the curve $r$ has no irregular equilibrium sequences.
\end{thm}

\begin{proof}
We prove the assertion for the case that $\kappa \rho + 1 > 0$; that is, the curve has a stable equilibrium at $r(0)$, for the case that $\kappa \rho +1<0$ we may apply a similar argument.
For every (sufficiently small) $\tau > 0$, define  the function $u$ by the implicit equation $\langle r(\tau), r(\tau+u) - r(\tau) \rangle = 0$.
We show that this function is strictly increasing, if $\tau$ is sufficiently small.

Clearly, $\lim\limits_{\tau \to 0} u(\tau) = 0$. Thus, as $r(\tau)$ is $C^3$-class, there is some $\varepsilon > 0$ such that if $0 < \tau, u < \varepsilon$, for some continuous vector function  $C(\tau,u) \in \Re^2$, we have
\begin{equation}\label{eq:Taylor}
r(\tau+u) = r(\tau) + \dot{r}(\tau) u + \frac{1}{2} \ddot{r}(\tau) u^2 + C(\tau,u) u^3
\end{equation}

Consider the $2$-variable function $F(\tau,u)= \langle r(\tau), r(\tau+u) - r(\tau) \rangle$. If at a point $(\tau,u)$, $F(\tau,u) = 0$, then
$\partial_u F (\tau,u) = \langle r(\tau), \dot{r}(\tau+u) \rangle < 0$, since  the angle of the two vectors is greater than $\frac{\pi}{2}$ if $\tau > 0$.
Since all partials of $F(\tau,u)$ are continuous, by the Implicit Function Theorem $u(\tau)$ is continuously differentiable.
Furthermore,
\[
u'(\tau) = -\frac{\partial_{\tau} F(\tau,u)}{\partial_u F(\tau,u)} = -\frac{\langle \dot{r}(\tau), r(\tau+u)-r(\tau)\rangle + \langle r(\tau), \dot{r}(\tau+u)- \dot{r}(\tau)\rangle}{\langle r(\tau), \dot{r}(\tau+u) \rangle} .
\]

We have observed that for the denominator in this equation, we have $\langle r(\tau), \dot{r}(\tau+u) \rangle < 0$ for every $0 < \tau < \varepsilon$, if $F(\tau,u)=0$.
On the other hand, substituting (\ref{eq:Taylor}) into the numerator, we obtain that
\[
\langle \dot{r}(\tau), r(\tau+u)-r(\tau)\rangle + \langle r(\tau), \dot{r}(\tau+u)- \dot{r}(\tau)\rangle = u \left( \langle \dot{r}(\tau), \dot{r}(\tau) \rangle + \langle \ddot{r}(\tau), r(\tau) \rangle \right) +  C^{*}(\tau,u) u^2,
\]
where  $C^{*}(u,\tau)$ is a continuous function.
Note that $ \langle \dot{r}(\tau), \dot{r}(\tau) \rangle + \langle \ddot{r}(\tau), r(\tau) \rangle = \ddot{\langle r(\tau) , r(\tau) \rangle} > 0$ if $\tau$ is sufficiently small. Thus, as $u(\tau) \to 0$ if $\tau \to 0$, we have $u'(\tau) > 0$ if $\tau$ is sufficiently small, which yields that on this interval $u(\tau)$ is strictly increasing.
This proves the existence of $K(\Delta)$. To show the existence of $k(\Delta)$, we may apply the same consideration for the case that $\tau < 0$.
\end{proof}



\begin{rem}\label{rem:no_irregular}
The proof of Theorem~\ref{thm:no_irregular} yields that the function $u(\tau)$ is non-decreasing even if $m$ is a degenerate equilibrium point. Thus, the assertion in Theorem~\ref{thm:no_irregular} holds even in this case.
\end{rem}

\begin{rem}\label{rem:no_irregular1}
Theorem~\ref{thm:no_irregular} implies that if we sum all imaginary equilibrium indices according to formulae (\ref{mainresult2D}-\ref{local2D}) then we obtain the global imaginary equilibrium index $N^0$ associated with the curve $r$, approximating the number of \emph{all} local equilibrium points associated with the polygon $r^{\Delta}$ in the $\Delta \to 0$ limit.
\end{rem}

\subsection{Equilibria on random meshes}\label{ss:random}

In this subsection we prove a probabilistic version of the formulae (\ref{mainresult2D}) which are the main planar result in \cite{Monatshefte}, and deals with an equidistant, $n$-element partition of the interval $[\underline{\tau},\overline{\tau}]$. Our goal is to show that even if the discretization is non-uniform, those formulae provide good estimates, thus the numbers predicted by the formulae may be observed in computer simulations.

\begin{thm}\label{thm:stochastic}
Let $r : [\underline{\tau},\overline{\tau}] \to \Re^2$ be a $C^4$-class curve satisfying the conditions in the beginning of Section~\ref{s:curves}.
For arbitrary $n \geq 2$ and $\delta \leq \min \{ |\underline{\tau}|,\overline{\tau} \}$ define the probability distribution $\zeta(n,\delta)$ in the following way:
Choose $n$ points $\tau_1, \tau_2, \ldots, \tau_n$ independently and using uniform distribution on $[-\delta,\delta]$.
Label the points such that $\tau_1 \leq \tau_2 \leq \ldots \leq \tau_n$, and set $Q^n = \bigcup_{i=1}^{n-1} [r(\tau_{i}),r(\tau_{i+1})]$.
Then $\zeta_{n,\delta}$ is defined by
\[
p(\zeta_{n,\delta} = k) = \hbox{ the probability that } Q^n \hbox{ has } k \hbox{ stable equilibria with respect to } o ,
\]
where $k=0,1,\ldots,n-1$.
Then, for every $\varepsilon > 0$ there is some $\delta= \delta(r,n,\varepsilon) > 0$ such that
\[
\left| E(\zeta_{n,\delta}) - \frac{1}{\lambda} \left( 1- \frac{1}{(1+\lambda)^n} \right) \right| < \varepsilon ,
\]
where $\lambda =|1+\kappa \rho|$.
\end{thm}

\begin{proof}
By the Implicit Function Theorem, the  coordinate function  $x(\tau)$ of $r$ is invertible in a neighborhood of  $\tau=0$, and thus, in this neighborhood the curve can be written as the graph of a function $y=f(x)$. Since the function $x=x(\tau)$ and its inverse are both $C^4$-class, a uniform distribution for $\tau$ on the interval $[-\delta,\delta]$ correspond to an `almost' uniform distribution of $x$ on the interval $[x(-\delta),x(\delta)]$ for sufficiently small values of $\delta$.  Thus, it suffices to prove the assertion for graphs of $1$-variable functions,  i.e. we may assume that $r$ is defined by $r(x) = (x,f(x))$ for some  $C^4$-class function $f$. Note  that then $\kappa = f''(0)$.

Now, choose some values $-\delta \leq x_1 \leq \ldots \leq x_n \leq \delta$ independently, and for  $i=1,2,\ldots, n$, let  $q_i=(x_i,f(x_i))$.
 Observe that there is an equilibrium point on the segment $[q_i,q_{i+1}]$ if, and only if $\langle q_i,q_{i+1}-q_i \rangle \leq 0 \leq \langle q_{i+1}, q_{i+1}-q_i \rangle$ (see also (3) in \cite{Monatshefte}).

In the following, we use the second-degree Taylor polynomial of $f$, by which we have
\[
f(x) = \rho + \frac{\kappa}{2} x^2 + C(x) x^3,
\]
where $C(x)$ is a  continuously differentiable function on $[-\delta,\delta]$, implying that there is some $K \in \Re$ such that $|C(x_2)-C(x_1)| \leq K |x_2-x_1|$ for every $x_1,x_2 \in [-\delta,\delta]$.
Then
\[
\left| \langle q_i,q_{i+1}-q_i \rangle - (x_{i+1}-x_i) \left( x_i + f(x_i) \frac{\kappa}{2}  (x_{i+1}+x_i) + C(x_{i+1})(x_{i+1}^2+x_ix_{i+1}+x_i^2)\right)\right| \leq
\]
\[
\leq K(x_{i+1}-x_i) x_i^3 .
\]
Thus, if $\delta$ is sufficiently small, the sign of  $\langle q_i,q_{i+1}-q_i \rangle$ is `almost always' equal to the sign of
$x_i+\frac{\kappa \rho}{2}(x_i+x_{i+1})$. We obtain similarly that the sign of $\langle q_{i+1}, q_{i+1}-q_i \rangle$ is
`almost always' equal to that of $x_{i+1}+\frac{\kappa \rho}{2}(x_i+x_{i+1})$.

Let  $\hat{\zeta}(n,\delta)$ denote the probability distribution where, choosing $n$ points $-\delta \leq x_1 \leq x_2 \leq \ldots \leq x_n \leq \delta$ uniformly and independently, $P(\hat{\zeta}(n,\delta) = k)$ is the probability that exactly $k$ pairs $\{x_i,x_{i+1} \}$ satisfy the inequalities
\begin{equation}\label{eq:conditions}
x_i+\frac{\kappa \rho}{2}(x_i+x_{i+1})< 0 \quad \mathrm{and} \quad x_{i+1}+\frac{\kappa \rho}{2}(x_i+x_{i+1}) > 0.
\end{equation}
By our previous argument, it suffices to prove that $E(\hat{\zeta}(n,\delta)) = \frac{1}{\lambda} \left( 1- \frac{1}{(1+\lambda)^n} \right)$.
To do it, we distinguish four cases depending on the value of $\kappa \rho$, and observe that convexity and the nondegeneracy of the equilibrium point implies that $0 > \kappa \rho \neq -1$.
These cases are $-1 < \kappa \rho < 0$, $-2 < \kappa \rho < -1$, $\kappa \rho = -2$, and $\kappa \rho > -2$.
We note that the computations in all these cases are almost identical, and thus, we carry them out only in the first case.

Accordingly, assume that $-1 < \kappa \rho < 0$.
We compute the probability  $p_i$ that (\ref{eq:conditions}) is satisfied for some fixed value of $i$.
Putting $\mu = -\frac{2 + \kappa \rho}{\kappa \rho} > 1$, in this case the inequalities in (\ref{eq:conditions}) are equivalent to $x_{i+1} \geq \mu x_i$ if $0 \leq x_i \leq \delta$, and $x_{i+1} \geq \frac{x_i}{\mu}$ if $-\delta \leq x_i \leq 0$.
Thus,  $p_i$ is equal to the fraction of the volumes of two regions. The region in the denominator is the simplex defined by the inequalities $-\delta \leq x_1 \leq x_2 \leq \ldots \leq x_n \leq \delta$; its volume is equal to $\frac{2^n \delta ^n}{n!}$. The region in the numerator is equal the union of two nonoverlapping regions, which are defined by the inequalities
\begin{enumerate}
\item[(8)] $0 \leq x_i \leq \frac{\delta}{\mu}$, $-\delta \leq x_1 \leq x_2 \leq \ldots \leq x_i$, $\mu x_i \leq x_{i+1} \leq \ldots, x_n \leq \delta$,
\item[(9)] $-\delta \leq x_i \leq 0$, $-\delta \leq x_1 \leq _2 \leq \ldots \leq x_i$ and $\frac{x_i}{\mu} \leq x_{i+1} \leq \ldots \leq x_n \leq \delta$.
\end{enumerate}
Hence,
\[
p_i = \frac{n!}{2^n \delta^n} \left( \int\limits_{-\delta}^0 \frac{(\delta+\tau)^{i-1}}{(i-1)!} \cdot \frac{\left(\delta-\frac{\tau}{\mu}\right)^{n-i}}{(n-i)!} \, d \tau +  \int\limits_{0}^{\frac{\delta}{\mu}} \frac{(\delta+\tau)^{i-1}}{(i-1)!} \cdot \frac{\left(\delta-\mu \tau\right)^{n-i}}{(n-i)!} \, d \tau \right)
\]
 By the linearity of expectation, we obtain that the expected value $E$ of the number of indices $i$ satisfying (\ref{eq:conditions}) is equal to $\sum\limits_{i=1}^{n-1} p_i$. Summing up and applying the Binomial Theorem, we have
\[
E =  \frac{n}{2^n \delta^n} \left( \int\limits_{-\delta}^0 \left( 2\delta +(1-\frac{1}{\mu}) \tau \right)^{n-1} - (\delta + \tau)^{n-1} \, d \tau +
 \int\limits_{0}^{\frac{\delta}{\mu}} \left( 2\delta + (1-\mu) \tau \right)^{n-1}-(\delta+\tau)^{n-1} \, d \tau \right), 
\]
from which an elementary computation yields that
\[
E = \frac{\mu + 1}{\mu -1} \left( 1 - \left( \frac{\mu+1}{ 2\mu} \right)^{n-1}\right) =
\frac{1}{\lambda}\left( 1- \frac{1}{(1+\lambda)^{n-1}} \right).
\]
\end{proof}

\begin{rem}
It is easy to check that the function $E(\lambda) = \frac{1}{\lambda}\left( 1- \frac{1}{(1+\lambda)^{n-1}} \right)$ is strictly decreasing on the interval $(0,\infty)$, and that $\lim\limits_{\lambda \to 0+0} E(\lambda) = n-1$.
\end{rem}


Finally, we show how one can reconstruct the number of equilibrium points of a smooth  plane curve $r$ from one of its sufficiently fine discretizations.
We state it  in a slightly different form, for graphs of functions. In our setting, the function $f$ in Remark~\ref{rem:2D} is the Euclidean distance function of $r$ from the given reference point.

\begin{rem}\label{rem:2D}
Let $f :  [\underline{x},\overline{x}] \to \Re$ be a $C^2$-class function with finitely many stationary points, each in the  open interval $(\underline{x},\overline{x})$, such that the second derivative of $f$ at each such point is not zero.  Assume that $f$ has $k$ local minima and $l$ local maxima in $(\underline{x},\overline{x})$. Let $\underline{x} = x_0 < x_1 < x_2 < \ldots < x_{n-1} < x_n = \overline{x}$ denote the division points of the equidistant $n$-element partition of $[\underline{x},\overline{x}]$.
Then, if $n$ is sufficiently large, there are exactly $k$ integers $0 < j < n$ satisfying $x_j < \min \{ x_{j-1}, x_{j+1} \}$, and $l$ integers $0 < j < n$ satisfying $x_j > \max \{ x_{j-1}, x_{j+1} \}$.
\end{rem}

\section{Dynamic theory: local equilibria on finely discretized, evolving planar curves}\label{ss:jump}

In this section, we deal with a $1$-parameter ($t$) family of closed convex curves $r_t(\tau)$, where $t \in [\underline{t},\overline{t}]$ is time, and $\tau \in [\underline{\tau},\overline{\tau}]$ is the spatial parameter. We assume that $r_t(\tau)$ is a $C^3$-class function of $\tau$, and this function, and all its derivatives with respect to $\tau$, depend continuously on $t$.

We denote the \emph{evolute} of the function $r_t$ by $E_t : [\underline{\tau},\overline{\tau}] \to \Re^2$.
We say that a point $E_t(\tau)$ of $E_t$ is \emph{general} if it is not a cusp, and for any $\tau' \neq \tau$, $E_t(\tau') \neq E_t(\tau)$. We say that $E_t$ is \emph{locally convex} at a general point $E_t(\tau)$ if $E_t(\tau)$ has a neighborhood $V$ in $\Re^2$ such that $E_t \cap V$ is a strictly convex curve. If $V$ satisfies this property, the convex, connected region of $(\Re^2 \setminus E_t) \cap V$ is called the \emph{convex side} of $E_t$ at $E_t(\tau)$, and the other region is called the \emph{concave side} of $E_t$.

\begin{thm}\label{thm:caustic}
Let $r_t:[\underline{\tau},\overline{\tau}] \to \Re^2$ be a $1$-parameter family of convex curves satisfying the conditions above.
Let $o: [\underline{t},\overline{t}] \to \Re^2$ be a continuous curve that transversely intersects $E_t$ at $E_{t^{\star}}(\tau_0) = p(t^{\star})$, and let $N(o(t))$ (resp.  $N^0(o(t))$) denote the number of global equilibrium points (the sum of the imaginary equilibrium indices) of $r_t$ with respect to $o(t)$. Assume that $E_{t^{\star}}$ is locally convex at the general point $E_{t^{\star}}(\tau_0)$.
\begin{itemize}
\item[(i)] If $o(t)$ moves from the convex side of $E_{t^{\star}}$ to its concave side as $t$ increases, then $N(o(t))$ increases by $2$ at $t=t^{\star}$, $\lim\limits_{t \to t^{\star}-0} N^0(o(t))= \mathrm{constant}$ and $\lim\limits_{t \to t^{\star}+0} N^0(o(t))= \infty$.
\item[(ii)] If $o(t)$ moves from the concave side of $E_{t^{\star}}$ to its convex side as $t$ increases, then $N(o(t))$ decreases by $2$ at $t=t^{\star}$, $\lim\limits_{t \to t^{\star}-0} N^0(p(t))= \infty$ and $\lim\limits_{t \to t^{\star}+0} N^0(p(t))= \mathrm{constant}$.
\end{itemize}
\end{thm}

As we remarked in the introduction, we call the events in (i) and in (ii) in the evolution  of $N^0(t)$ a \emph{$C$-type} and an \emph{$A$-type} event, respectively. The same events in the evolution of $N(t)$ correspond to a generic, codimension one saddle-node bifurcation and they are called in  bifurcation theory\emph{creation} and \emph{annihilation}, respectively \cite{Poston}.

Theorem~\ref{thm:caustic} is an immediate consequence of the following lemma, Theorem~\ref{thm:no_irregular}, and Theorem 1 of \cite{Monatshefte}.

\begin{lem}\label{lem:caustic}
Let $r:[\underline{\tau},\overline{\tau}] \to \Re^2$ be a  $C^3$-class closed, convex curve. Let the evolute of the curve be $E : [\underline{\tau},\overline{\tau}] \to \Re^2$.
Assume that $E$ is locally convex at a point $E(\tau_0)$. Let  $N(q)$ denote the number of equilibrium points of $r$ with respect to $q$. Then $E(\tau_0)$ has a neighborhood $V$ such that  $N(q) + 2 = N(q')$ for any point $ q \in V$ on the convex, and any point $ q' \in V$ on the concave side of $E$ at $E(\tau_0)$.
\end{lem}

\begin{proof}
Without loss of generality, throughout the proof we consider only spherical neighborhoods of $E(\tau)$.
We recall the well-known fact that moving $q$ continuously, $N(q)$ changes if, and only if $q$ crosses the evolute $E$ \cite{robust, Poston}. Since for any sufficiently small neighborhood $V$ of $E(\tau_0)$, $E$ intersects $V$ in a simple curve, we have that for any $q \in V \setminus E$, $N(q)$ depends only on which side of $E$ $q$ is located. Thus, it suffices to prove that for some $q$ on the convex, and some $q'$ on the concave side of $E$, we have $N(q)+2=N(q')$.

Let  $L(\tau)$ denote the normal line of $r$ at $r(\tau)$, i.e. the line perpendicular to $\dot{r}(\tau)$ and passing through $r(\tau)$.
Note that the number of equilibria with respect to a point $z$ is the number of normal lines of $r$ passing through $z$.
Let the normal lines of $r$ passing through $E(\tau_0)$ be $L(\tau_0), L(\tau_1), \ldots, L(\tau_k)$. Let $q$ and $q'$ be points sufficiently close to $E(\tau_0)$ such that $q$ is on the convex and $q'$ is on the concave side. Then, apart from some small neighborhood of $\tau_0$, there are exactly $k$ normals (say $L(\tau_i')$, where $i=1,2,\ldots,l$) passing through $q$, and $k$ normals (say $L(\tau''_i)$, $i=1,2,\ldots,k$) passing through $q'$, where $\tau'_i$ and $\tau''_i$ are `close' to $\tau_i$. On the other hand, there are exactly two lines through $q'$ that touch $E$ near $E(\tau_0)$, and since $E(\tau_0)$ is a general point of $E$, these two lines are normal lines of $r$ at exactly two values $\tau^*_1$ and $\tau^*_2$, close to $\tau_0$. Since every normal line of $r$ is tangent to $E$, it follows that $N(q')=k+2$. Similarly, there are no lines through $q$ that touch $E$ close to $E(\tau_0)$, yielding that $N(q) = k$.
\end{proof}

\begin{rem}\label{rm:caustic1}
For any fixed, small $\Delta$, the lines, normal to a side of the approximating polygon, and passing through a vertex of the side, decompose the plane into pieces of small diameters. The union of these normals, which we may call the \emph{evolute of the polygon}, has the property that the number of the equilibria of the polygon changes if, and only if the reference point crosses this set \cite{robust}. Thus, even though imaginary equilibrium indices change continuously during a continuous motion of the reference point, the quantity $N^{\Delta}(t)$ makes rapid jumps during this motion. This phenomenon can be observed, i.e. on Figure \ref{fig:coev}, especially when the reference point approaches the evolute of the curve.
\end{rem}

\begin{rem}\label{rm:caustic2}
Let the vertices of the approximating polygon be $p_1, p_2, \ldots, p_n$ in counterclockwise order, and let us orient the two normals to the side $[p_j,p_{j+1}]$ at $p_j$ and $p_{j+1}$ such that the part of the normal through $p_i$ in the polygon points towards $p_i$, and the part of the other normal in the polygon points away from $p_{i+1}$. In this case, similarly like in Theorem~\ref{thm:caustic}, we can determine how the number of equilibria of the polygon changes if we cross its evolute: it increases if we cross an oriented line in it from right to left, and decreases in the opposite case. We note that the boundary of a cell in the cell decomposition defined by the evolute of the polygon is cyclic if, and only if the number of equilibria has a local extremum in this region.
\end{rem}


\subsection{The size of Critical Flocks}\label{ss:size}

In this subsection we examine the number of local equilibria of the one-parameter family of curves $r_t$ at some fixed time $t$, exactly when the reference point  $o(t)=o$ is on the evolute of the curve, so we will drop the subscript
of $r$. Strictly speaking, this subsection could also be regarded of the \emph{static theory} of equilibria, discussed in Section \ref{s:curves}, however, its content is  more closely related to the subject of the current section.
More specifically, we use the following notation.

Let $r:[\underline{\tau},\overline{\tau}] \to \Re^2$ be $C^{2k-1}$-class plane curve which has a unique equilibrium point $m = r(0)$ with respect to the origin $o$, where $k \geq 2$ is the order of the first nonzero derivative of the Euclidean distance function $\tau \mapsto |r(\tau)|$ at $\tau = 0$. As in Section~\ref{s:curves}, let $F^{\Delta}$ be the $n$-element equidistant partition of $[\underline{\tau},\overline{\tau}]$ with $\Delta = \frac{\overline{\tau} - \underline{\tau}}{n}$. Let $N^{\Delta}(o)$ denote the number of equilibrium points of the approximating polygon defined by $F^{\Delta}$.

\begin{thm}\label{thm:critical}
We have $N^{\Delta}(o) = \Theta\left( {\Delta}^{-\frac{k-2}{k-1}} \right)$, i.e. there are constants $C_1, C_2 \in \Re$ such that $C_1 {\Delta}^{-\frac{k-2}{k-1}}  
\leq N^{\Delta}(o) \leq C_2 {\Delta}^{-\frac{k-2}{k-1}}$ holds for all $\Delta > 0$.
\end{thm}

\begin{rem}
The proof of Theorem~\ref{thm:critical} shows a slightly stronger statement, namely that the \emph{diameter} of the flock at $m$ is of order $\Theta
\left( \Delta^{\frac{1}{k-1}}\right)$, i.e. there are constants $c_1,c_2$ such that for any division point $\tau$ of $F^{\Delta}$ with $\left| \tau \right| \leq c_1 \Delta^{\frac{1}{k-1}}$, both $r(\tau)$ and the segment $[r(\tau),r(\tau + \Delta)]$ contains equilibrium points, and if $\left| \tau \right| \geq c_2 \Delta^{\frac{1}{k-1}}$, then neither.
\end{rem}
                                              
\begin{proof}
Using the formula $\Delta = \frac{\overline{\tau} - \underline{\tau}}{n}$, we need to prove that $N^{\Delta}(o) = \Theta\left( {n}^{\frac{k-2}{k-1}} \right)$.
Let $p_i = r(i \Delta)$ for any division point $i\Delta$ of $F^{\Delta}$. Note that if $p_i$ and $p_{i+1}$ are unstable equilibrium points, then $[p_i,p_{i+1}]$ contains a stable equilibrium, and if $[p_{i-1},p_i]$ and $[p_i,p_{i+1}]$ contain stable equilibria, then $p_i$ is an unstable equilibrium.
Thus, it suffices to prove the existence of constants $C_1,C_2 > 0$ such that if $|i| \leq C_1 n^{\frac{k-2}{k-1}}$, then $[p_i, p_{i+1}]$ contains a stable equilibrium point, and if $|i| > C_2 n^{\frac{k-2}{k-1}}$, then it does not. Clearly, in the proof we may assume that $n$ is sufficiently large.

We prove Theorem~\ref{thm:critical} only under the additional assumption that $\tau$ denotes the polar angle in a suitable polar coordinate system.
We remark that in the general case, in a suitable coordinate system, the polar angle $\phi$ can be expressed as a $C^2$-class function of $\tau$ in a neighborhood of $\tau = 0$, and in a small neighborhood of $0$, we have $| \phi(\tau) - \gamma \tau | \leq \theta \tau^2$ for some suitable $\gamma, \theta > 0$. Using this inequality, our argument can be modified for a general parametrization of $r$ in a straightforward way.
Under the assumption that $\tau$ denotes polar angle, the curve can be written as $r(\tau) = (\rho(\tau) \cos \tau, \rho(\tau) \sin \tau)$, where $\rho(\tau)$ is the $C^{2k-1}$-class positive distance function $\rho(\tau) = |r(\tau)|$.

Similarly as before, we observe that there is a stable point on $[p_i,p_{i+1}]$ if, and only if
$\langle p_i,p_{i+1}-p_i \rangle < 0 < \langle p_{i+1},p_{i+1}-p_i \rangle$. In our notation, these inequalities are equivalent to
\begin{equation}\label{eq:polar}
\rho\left( (i+1) \Delta \right) - \rho\left( i \Delta \right) \cos \Delta > 0\,\, \textrm{and} \,\,
\rho\left( i \Delta \right) - \rho\left( (i+1) \Delta \right) \cos \Delta > 0 .
\end{equation}

First, by the differentiability properties of $r$, there are constants $A_k,A_{k+1},\ldots, A_{2k-2}$ and $\rho, B > 0$ such that
\begin{equation}\label{eq:taylor}
\left| \rho(\tau) - \rho -\sum_{j=k}^{2k-2} A_j \tau^j  \right| < B \tau^{2k-1}.
\end{equation}

We carry out the computations only for positive indices $i$, and assuming that $A_k < 0$; that is, assuming that for small positive values of $\tau$, $\rho(\tau)$ is a decreasing function. Under these conditions, the second inequality is satisfied for all values of $i > 0$.

Assume now that the first inequality is satisfied. Using the inequalities $1- \frac{\Delta^2}{2} \leq \cos \Delta \leq 1- \frac{\Delta^2}{2} + \frac{\Delta^4}{24}$, we obtain that then
\[
0 < \rho\left( (i+1) \Delta \right) - \rho\left( i \Delta \right) + \frac{\Delta^2}{2} \rho\left( i \Delta \right).
\]
Now, let us estimate the right-hand side from above by (\ref{eq:taylor}). Then, applying the inequalities $ s x^{s-1} \leq (x+1)^s - x^s \leq s (x+1)^{s-1}$ for all real $x > 1$ and integer $s \geq 2$, we obtain that
\[
0 < A_k k \Delta^k i^{k-1} + \sum_{j=k+1}^{2k-2} |A_j| j \Delta^j i^{j-1} + 3B (i+1)^{2k-1} \Delta^{2k-1} +
\frac{\Delta^2 \rho}{2} +  \frac{\Delta^2 }{2} \sum_{j=k}^{2k-2} A_j \left( i \Delta \right)^j,
\]
where $\rho = \rho(0) = |m|$.
 Let $i = C_2 {\Delta}^{-\frac{k-2}{k-1}}$, where $C_2$ will be chosen later. Then an elementary consideration shows that the largest member of the above expression, in terms of $n$, is of order $\Delta^2$, and  by the inequality $A_k < 0$, we have that if $\Delta$ is sufficiently small, then
$A_k k C_2^{k-1} + \frac{\rho}{2} > 0.$ Clearly, for suitably chosen values of $C_2$ this is a  contradiction, which, by Theorem~\ref{thm:no_irregular} and Remark~\ref{rem:no_irregular}, proves that for any $i > C_2 {\Delta}^{-\frac{k-2}{k-1}}$, the segment $[p_i,p_{i+1}]$ contains no equilibrium point.

To prove the existence of a value $C_1 > 0$ such that the first inequality in (\ref{eq:polar}) is satisfied for all $0 < i <  C_1 n^{\frac{k-2}{k-1}}$, we may apply a similar consideration.
\end{proof}

\begin{rem}
We observe that our result implies that if $m$ is a nondegenerate equilibrium point (i.e. $k=2$), then there are constants $C_1,C_2 >0$ such that $C_1 <  N^{\Delta}(o) < C_2$ for all values of $\Delta$. A stronger version of this result was proven in \cite{Monatshefte}.
\end{rem}

\subsection{Time evolution of Critical Flocks}\label{ss:time}

In Theorem~\ref{thm:caustic} we have seen that if the reference point transversely crosses the evolute of a curve at a generic point $E(\tau)$, then the number $N^0$ of local equilibria tends to infinity 
 if the reference point approaches $E(\tau)$ from its convex side (annihilation) and remains constant if the reference point approaches $E(\tau)$ from it concave side (creation).
In this subsection we explore more thoroughly these limits, and also cases when the point $E(\tau)$ is degenerate.
Note that if $E(\tau)$ is a generic point then the derivative of the curvature of $r$ corresponding to this point is not zero.
If $E(\tau)$ is a cusp, we assume that it is a \emph{generic cusp}, i.e. that  here the second derivative of the curvature of $r$ is not zero.

\begin{thm}
Let $r$ be a $C^4$-class plane curve with a unique, degenerate equilibrium point $m$ with respect to the origin $o$. Let $ o(t) = (\mu t, \nu t)$, where $t \in [-\varepsilon,\varepsilon]$ for some small value of $\varepsilon > 0$. Let $N^0(t)$ denote the number of local equilibria with respect to $o(t)$.
Then the following holds.
\begin{itemize}
\item[(i)] If the curve $o(t)$ crosses the evolute of $r$ at $o$, and the evolute $E$ is locally convex at $o$ such that for $t > 0$, $o(t)$ is on the concave side of $o$, then $N^0(t) = 0$ as $t \to 0-0$, and $N^0(t) = \Theta\left( \frac{1}{\sqrt{t}} \right)$ as $t \to 0+0$.
\item[(ii)] If $o(t)$ is tangent to the evolute of $r$ at $o$, and $E$ is locally convex at $o$, then $N^0(t) = \Theta\left( \frac{1}{t} \right)$ as $t \to 0$.
\item[(iii)] If $o(t)$ is not tangent to $E$ at the general cusp $o$, then $N^0(t) = \Theta\left( \frac{1}{\sqrt[3]{t^2}} \right)$ as $t \to 0$.
\item[(iv)] If $o(t)$ is tangent to $E$ at the general cusp $o$, then $N^0(t) = \Theta\left( \frac{1}{t} \right)$ as $t \to 0$.
\end{itemize}
\end{thm}

\begin{proof}
Note that since the number of local equilibria is independent of the parametrization of $r$, we may assume that $r$ is given in the form $r(x) = (x,f(x))$ for some $C^4$-class function $f$ and for $x \in [-a,a]$. Let $\rho(x,t)=|r(x) - o(t)|$, and let $\kappa(x)$ denote the signed curvature of $r$ at the point $r(x)$. Let $X(t)$ denote the set of values of $x$ such that $r(x)$ is an equilibrium point with respect to $o(t)$.
Then $N^0(t)= \sum_{x \in X(t)} \frac{1+|\rho(x,t) \kappa(x)|}{|1+\rho(x,t) \kappa(x)|}$.
Note that since $m=r(0)$ is a degenerate equilibrium point with respect to $o$, we have $\kappa(0) = - \frac{1}{\rho(0,0)} < 0$, and thus, we may restrict our investigation to a neighborhood of $x=0$ where $\kappa(x)$ (and so $f''(x)$) is negative. From this an elementary computation yields that
$N^0(t) = \sum_{x \in X(t)} \frac{2}{|1+\rho(x,t) \kappa(x)|}$.

For any value of $x$, let us define the function $t(x)$ by the condition that $r(x)$ is an equilibrium point with respect to $o(t)$.
Note that this condition is equivalent to saying that $\langle r'(x), r(x)-o(t)\rangle = 0$.
Observe that if the line $L = \{ o(t) : t \in \Re \}$ is not perpendicular to the vector $r'(x)$, then this equation has a unique solution for $t$.
It is an elementary computation  to show that the condition that $L$ is not perpendicular to $r'(x)$ is equivalent to the condition that $\mu + \nu f'(x) \neq 0$. Under this condition, the unique solution $t$ can be expressed as $t(x) = \frac{x+f'(x)f(x)}{\mu + \nu f'(x)}$.
Note that if $\mu+\nu f'(x) = 0$ for some $x$ arbitrarily close to $x=0$ (but not equal to zero) and $\nu \neq 0$, then it would contradict our assumption that $f'(x)$ is differentiable at $x=0$, and $f''(0) \neq 0$. If $\nu = 0$, then $\mu \neq 0$, and thus, $\mu+\nu f'(x) \neq 0$ unless $f'(x) = 0$. Thus, we may assume that $t(x)$ uniquely exists unless $\mu = 0$ and $x=0$.

First, we assume that $\mu \neq 0$, or $\mu = 0$ and $x \neq 0$.
Let us define the function $F(x) = 1 + \rho(x,t(x)) \kappa(x)$. We examine the first nonvanishing term of this function as a function of $x$. Let $\rho_0 = f(0) > 0$.
Note that as $f$ is a $C^4$-class function, there is some number $A \in \Re$, and continuous functions $B,C,D$ such that
$f(x) = \rho_0 - \frac{1}{2 \rho_0}x^2+Ax^3+B(x) x^4$, $f'(x) = - \frac{1}{\rho_0}x+3Ax^2+C(x) x^3$, $f''(x) = -\frac{1}{\rho_0} + 6Ax+D(x) x^2$.
Here, it is easy to see that whereas $B,C,D$ may be different functions, we have $f^{(4)}(0) = 2D(0) = 6C(0) = 24B(0)$.

On the other hand, by the formula $\kappa(x) = \frac{f''(x)}{\left( 1+f'^2(x)\right)^{\frac{3}{2}}}$ and an elementary computation, we obtain that $\kappa'(x) \neq 0$ is equivalent to $A \neq 0$, and $\kappa'(0)=0$ and $\kappa''(x) \neq 0$ is equivalent to $A=0$ and $f^{(4)}(0) \neq - \frac{3}{\rho_0^3}$.
Note that $F(x) \to 0$ as $x \to 0$.
An elementary algebraic transformation shows that
\[
F(x) = \frac{(1+f'^2(x))^3 - f''^2(x) \left( (x-\mu t(x))^2 + (f(x) - \nu t(x))^2\right)}{(1+f'^2(x))^3 \left( 1 - \kappa(x) \rho(x,t(x))\right)}.
\]

Consider the case that $\mu \neq 0$. If $A \neq 0$, then $F(x) = 6 \rho_0 A x + G(x) x^2$ and $t(x) = \frac{3A}{\mu}x^2 + H(x) x^3$ for some continuous functions $G, H$. If $A= 0$ and $f^{(4)}(0) \neq - \frac{3}{\rho_0^3}$, then $F(x) = \frac{\rho_0^4 \left( 2 D(x) \rho_0^3 + 3\right)}{2} x^2 + G(x) x^3$ and $t(x) = \frac{1}{2 \mu \rho_0^2} x^3 + H(x) x^4$ for some continuous functions $G, H$. Here we note that $D(0) \rho_0^3 + 3 = f^{(4)}(0) \rho^3 + 3 \neq 0$.

If $\mu = 0$ and $A \neq 0$, then $F(x) = 3 \rho_0 A x + G(x) x^2$ and $t(x) = -\frac{3A\rho_0^2}{\nu} x + H(x) x^2$ for some continuous $G$ and $H$. If $\mu = 0$, $A = 0$ and $f^{(4)}(0) \neq - \frac{3}{\rho_0^3}$, then $F(x) = \left( C(x) - D(x) - \frac{1}{\rho_0^3} \right) x^2 + G(x) x^3$, and $t(x) = - \frac{1}{2 \rho_0 \nu}x^2 + H(x) x^3$ for some continuous $G, H$. Note that $C(0) - D(0) = - \frac{f^{(4)}(0)}{3} \neq \frac{1}{\rho^3}$.

Finally, in the degenerate case, when $x = 0$ and $\mu = 0$, there is an equilibrium point for every value of $t$, and thus,  $|1 + \kappa(0) \rho(0,t)| = \frac{|\nu|}{\rho_0} |t|$.

To finish the proof one needs only to collect the number of equilibria in each case, and express all magnitudes in terms of $t$.
\end{proof}

\subsection{A numerical example for curvature-driven flow demonstrating the singular limit for local equilibria}\label{ss:num}
Here we demonstrate the results of the previous two subsections, in particular the size of the critical flock and Theorem \ref{thm:caustic} in the case where the curve is evolving under the curve shortening flow \cite{Grayson}.  In compact notation this flow may be written for a convex, embedded curve as
\begin{equation}\label{curv}
v=c\kappa,
\end{equation}
where $v$ is the speed in the direction of the inward normal and $\kappa$ is the curvature and $c$ is a scalar coefficient.
 This evolution is one of the most interesting evolutions from the point of view of geophysics since we know \cite{Grayson} that it only generates annihilations for $N(t)$ so, based on Theorem \ref{thm:caustic} we expect to see only $C$-type events in the evolution of $N^{\Delta}(t)$. In order to integrate equation (\ref{curv}), below
we describe a numerical scheme which produces in each time step an equidistant discretization with respect to the arc-length of the curve. $N^{\Delta}(t)$ is computed on the polygon determined by the vertices of that discretization. $N(t)$ is simply the number of extrema of the piecewise linear function determined by the vertex distances from the centroid in polar coordinates.

Let $I=[0,1]$; a smooth, non-intersecting curve with a natural parametrization and unit perimeter is denoted by $r(\tau): I\rightarrow\Re^2$. Let $(\cdot)'$ and $\dot{(\cdot)}$ stand for derivation with respect to the parameter of the curve and time, respectively. Then the curvature of the curve is simply $\kappa(\tau)=\left\|r(\tau)''\right\|$. The unit normal to the curve is $n(\tau)=r''(\tau)\kappa(\tau)^{-1}$. In case of the curve shortening flow (\ref{curv}), during time $dt$, the curve $r(\tau)$ is mapped to $\tilde{r}(\tau)$ via
\begin{equation}
\tilde{r}(\tau)=r(\tau)+c\kappa(\tau)n(\tau)dt+\mathcal{O}(dt)=r(\tau)+cr''(\tau)dt+\mathcal{O}(dt),
\end{equation}
\noindent Note that $\tilde{r}(\tau)$ is parametrized with respect to the arc length of $r(\tau)$, hence its parametrization is not natural. We aim to describe the curve by the tangent direction $\alpha(\tau)$. Hence $r'(\tau)=:e(\alpha(\tau))=[\sin(\alpha(\tau)),\cos(\alpha(\tau)]$. By keeping the linear terms in $dt$, for $\tilde{r}(\tau)$ we obtain $\left\|\tilde{r}'(\tau)\right\|=1-c\alpha'(\tau)dt$. Obviously
\begin{equation}
\label{eq:tangent:evol}
e(\alpha(\tau)+\dot{\alpha}(\tau)dt)=\frac{\tilde{r}'(\tau)}{\left\|\tilde{r}'(\tau)\right\|}
\end{equation}
\noindent holds for the mapped curve.  Taylor expansion of the left hand side of Equation (\ref{eq:tangent:evol}), substitution of $e(\alpha(\tau))$ into the derivative of $\tilde{r}'(\tau)$, the chain rule and the Frenet-formulas yield
\begin{equation}
e(\alpha(\tau))+e'(\alpha(\tau))\dot{\alpha}(\tau)dt+\mathcal{O}(dt)=\frac{e(\alpha(\tau))-ce(\alpha(\tau))\alpha'(\tau)dt+ce'(\alpha(\tau))\alpha''(\tau)dt}{1-c\alpha'(\tau)dt}.
\end{equation}

\noindent Algebraic manipulations and neglecting the $\mathcal{O}(dt)$ terms leaves
\begin{equation}
\dot{\alpha}(\tau)=c\alpha''(\tau).
\end{equation}

This simple, linear PDE can easily be simulated by a finite difference scheme for the spatial, and an Euler scheme for the time derivatives, respectively. However, since $\tau$ is not a natural parameter for $\tilde{r}(\tau)$, in each time-step the curve must be reparametrized to obtain $\tilde{r}(\tilde{\tau})$, a curve with a natural parametrization $\tilde{\tau}$.

\begin{figure}
\includegraphics[width=0.95\textwidth]{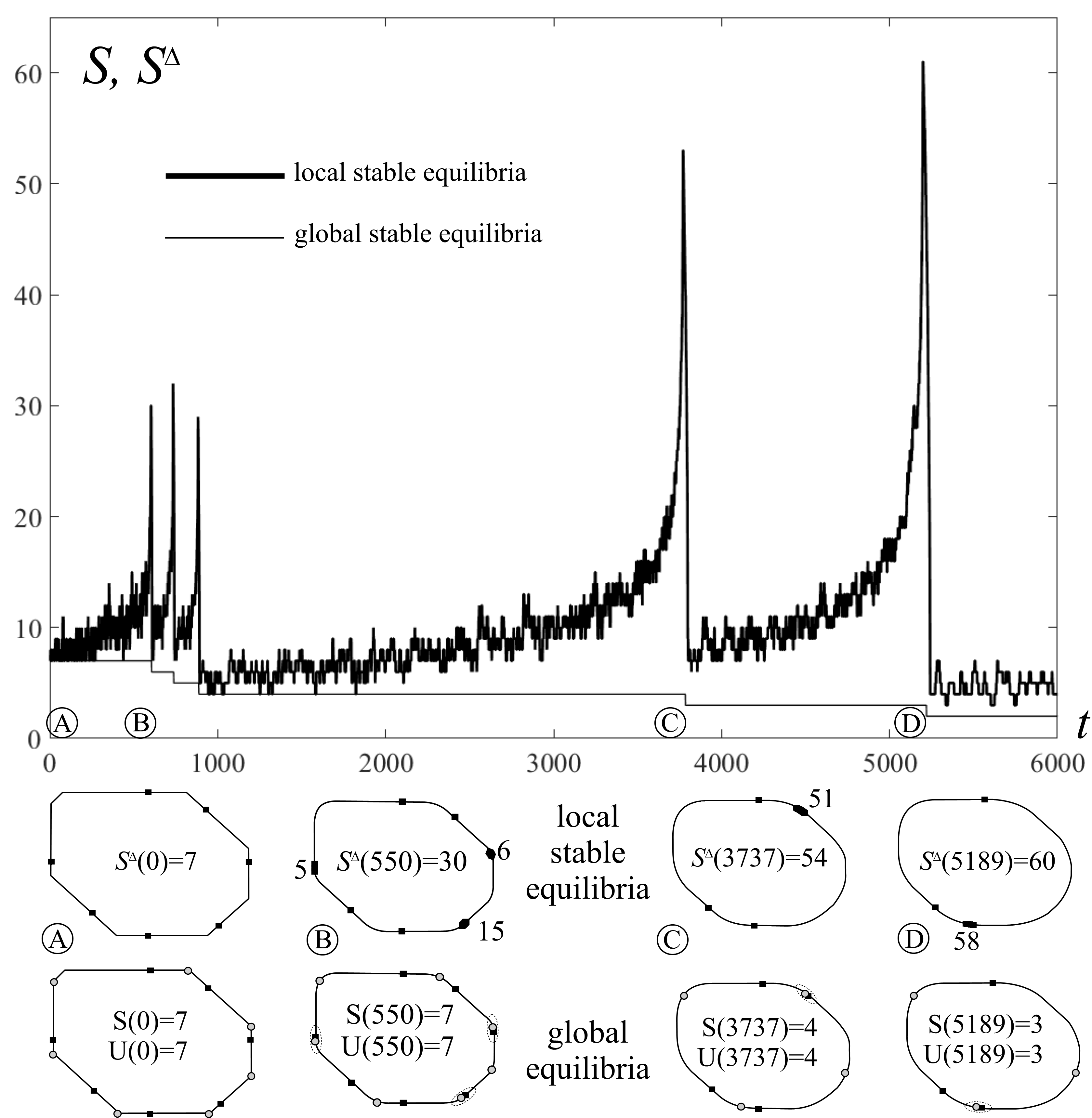}
\caption[]{Co-evolution of local and global equilibria in 2D under the curve-shortening flow (\ref{curv}). A polygon with $7$ stable and $7$  unstable balance points is  evolved in such a manner that the perimeter of the shape is kept unit. The simulation is carried out until $2$ stable and $2$ unstable (global) equilibria remain.}
\label{fig:coev}
\end{figure}

Figure \ref{fig:coev} shows the co-evolution of the local and global equilibria from a generic polygonal shape under the curve shortening flow. Observe the five type A events during the evolution.

\section{Smooth surfaces and their discretizations}\label{sec:3D}

\subsection{Earlier results}\label{ss:3Dmonatshefte}

Let $D = [\underline{u},\overline{u}] \times [\underline{v},\overline{v}]$, and let $r:  \to \Re^3$ be a convex, $C^3$-class surface having a unique, non-degenerate equilibrium point at $m=r(0,0)$ with respect to the origin $o$. Let $F^{\Delta}$ denote the $n \times n$ equidistant partition of $D$, where $\Delta = \left( \Delta_1, \Delta_2 \right) = \left( \frac{\overline{u}-\underline{u}}{n}, \frac{\overline{v}-\underline{v}}{n}\right)$. If $\left( i \Delta_1, j \Delta_2 \right)$ is a division point of $F^{\Delta}$, we call $(i,j)$ the \emph{indices} of the point $p^{\Delta}_{i,j}=r \left( i \Delta_1, j \Delta_2 \right)$. Similarly like in the planar case in Section~\ref{s:curves}, we note that in general, $i$ and $j$ are not integers but real numbers congruent $\mod 1$.

We define the polyhedral surface $P^{\Delta}$ in the following way:
\begin{itemize}
\item The vertices of $P^{\Delta}$ are exactly the points $p^{\Delta}_{i,j}$.
\item For any $i,j$, if the segment $[p^{\Delta}_{i,j},p^{\Delta}_{i+1,j+1}]$ is an edge of $\conv \{ o, p^{\Delta}_{i,j},p^{\Delta}_{i+1,j} ,p^{\Delta}_{i+1,j+1}, p^{\Delta}_{i,j+1}\}$, then the triangles 
$\conv \{ p^{\Delta}_{i,j},p^{\Delta}_{i+1,j} ,p^{\Delta}_{i+1,j+1} \}$ and $\conv \{ p^{\Delta}_{i,j},p^{\Delta}_{i,j+1} ,p^{\Delta}_{i+1,j+1} \}$, and otherwise the triangles
$\conv \{ p^{\Delta}_{i,j},p^{\Delta}_{i+1,j} ,p^{\Delta}_{i,j+1} \}$ and $\conv \{ p^{\Delta}_{i+1,j+1},p^{\Delta}_{i,j+1} ,p^{\Delta}_{i+1,j} \}$ are faces of $P^{\Delta}$.
\end{itemize}
Then $P^{\Delta}$ is a triangulated surface defined by the partition $F^{\Delta}$ of $D$.

If a face/edge of $P^{\Delta}$ contains an equilibrium point with respect to $o$, and $i,j$ are the minimum of the indices of the vertices of this face/edge, then we say that the indices of the equilibrium points are $(i,j)$. For any $K > 0$, we denote the number of stable/unstable/saddle points of $P^n$, whose indices satisfy the inequalities $|i|, |j| \leq K$, by $S^{\Delta}(K), U^{\Delta}(K)$ and $H^{\Delta}(K)$, respectively.

In \cite[Theorem 2]{Monatshefte}, the authors proved that if $K$ is sufficiently large and $\Delta$ is sufficiently small, then the quantities $S^{\Delta}(K), U^{\Delta}(K)$ and $H^{\Delta}(K)$ fluctuate around specific values $S^0(K)$, $U^0(K)$ and $H^0(K)$, whose values are
\begin{equation} \label{mainresult}
S^0(K)=d, \hspace{1cm} U^0(K)=\kappa_{1}\kappa_{2}\rho^2d, \hspace{1cm}  H^0(K)=-(\kappa_{1}+\kappa_{2})\rho d
\end{equation}
where $d=\frac{1}{|(\kappa_{1}\rho+1)(\kappa_{2}\rho+1)|}$, $\rho = |m|$, and $\kappa_1,\kappa_2 \leq 0$ are the (signed) principal curvatures of $r$ at $m$.

\subsection{Enumeration of global equilibria on finely discretized surfaces} \label{ss:global_3d}

The aim of this subsection is to find a $3$-dimensional analogue of Remark~\ref{rem:2D}; namely, given a fine discretization of a smooth surface, we intend to find the number of global equilibrium points of the surface.

Before stating our results, we need some preparation. We note that, for the sake of simplicity, instead of $\Delta$, to do this we describe an equidistant partition by the number  $n$ of intervals in it, instead of the size $\Delta$ of these intervals.
Let $f : [0,a] \times [0,b] \to \Re$ be a $C^3$-class function.
Consider a partition $F_n$ of the rectangle $D=[0,a] \times [0,b]$ into $n \times n$ congruent rectangles.
We call the vertices of these rectangles \emph{grid vertices}, and denote the grid vertex $\left( \frac{i}{n}a, \frac{j}{n}b \right)$ by $p_{i,j}$.
The \emph{neighbors} of the grid vertex $p_{i,j}$ are the four grid vertices $p_{i \pm 1,j}$ and $p_{i,j \pm 1}$.
The two pairs $p_{i \pm 1,j}$ and $p_{i,j \pm 1}$ are called \emph{opposite neighbors} of $p_{i,j}$.
A grid vertex $p$ is \emph{stationary}, if for any opposite pair $\{ q, q' \}$ of its neighbors, $f(p) \geq \max \{ f(q), f(q') \}$ or
$f(p) \leq \min \{ f(q), f(q') \}$ is satisfied.
If $p_{i,j}$ is a grid vertex, then the \emph{grid circle of center $p_{i,j}$ and radius $r$} is the set
\[
C_r(p_{i,j}) = \{ p_{l,m} : \max \{ |l-i|, |m-j| \} \leq r \}.
\]

During the consideration, we assume that $f$ has finitely many stationary points, each in the interior of the domain $D$, and the determinant of the Hessian of $f$ at each of them is nonzero. We assume that the grids we use are non-degenerate; more specifically, that if $p \neq p'$ are two grid vertices, then $f(p) \neq f(p')$.


\begin{thm}\label{thm:auxiliary}
Let $p=(x_0,y_0) \in \inter D$.
\begin{enumerate}
\item[(1)] If $p$ is \emph{not} a stationary point of $f$, then $p$ has a neighborhood $U \subset D$ such that for any $n \geq 1$, if $p_{i,j}$ and all its neighbors are contained in $U$, then $p_{i,j}$ is not a stationary grid vertex.
\item[(2)] If $p$ is a local minimum of $f$, then $p$ has a neighborhood $U$ and some sufficiently large value of $r$ such that for every sufficiently large $n$, there is exactly one grid vertex $p_{i,j}$ in $U$, which is minimal within its grid circle $C_r(p_{i,j})$.
\item[(3)] If $p$ is a local maximum of $f$, then $p$ has a neighborhood $U$ and some sufficiently large value of $r$ such that for every sufficiently large $n$, there is exactly one grid vertex $p_{i,j}$ in $U$, which is maximal within its grid circle $C_r(p_{i,j})$.
\item[(4)] If $p$ is a saddle point of $f$, then every neighborhood of $p$ contains a stationary grid vertex, and $p$ has a neighborhood $U$ and some sufficiently large value of $r$ such that for every sufficiently large $n$, any grid vertex $p_{i,j}$ in $U$ is neither maximal, nor minimal within its grid circle $C_r(p_{i,j})$.
\end{enumerate}
\end{thm}

\begin{proof}
First, we prove (1).
Let $L$ be the line through the origin and perpendicular to $\grad f(p)$, and note that the derivative of $f$  at $p$ is zero in this direction.
Let $\varepsilon > 0$ be sufficiently small, and $A$ be the union of the lines, through $p$, the angles of which with $L$ is not greater than $\varepsilon$.
Note that by the continuity of $\grad f$, $p$ has a neighborhood $U$ such that for any $q \in U$, $\grad f (q)$ is perpendicular to some line in $A$.
This implies that if $q \in U$, and $A$ contains no line parallel to the vector $u$, then $f'_u(q) \neq 0$. Without loss of generality, we may assume that $U$
is a Euclidean disk in $\Re^2$.

Now, consider any division $F_n$, and assume that the grid vertex $p_{i,j}$ and all its neighbors are contained in $U$.
Since $\varepsilon > 0$ is sufficiently small, the $x$-axis or the $y$-axis is not parallel to any line in $A$.
Without loss of generality, let the $x$-axis have this property.
We show that the sequence $f(p_{i-1,j}), f(p_{i,j})$ and $f(p_{i+1,j})$ is strictly monotonous.
Indeed, if, for example, $f(p_{i,j}) \geq \max \{ f(p_{i-1,j}), f(p_{i+1,j})\}$, then by the Lagrange Theorem, for some $q_1,q_2 \in U$, we have
$f'_x(q_1) \leq 0 \leq f'_x(q_2)$, which, by the continuity of $f'_x$, yields that  for some $q \in U$, we have  $f'_x(q) = 0$, which contradicts the definition of $A$.
If $f(p_{i,j}) \leq \min \{ f(p_{i-1,j}), f(p_{i+1,j})\}$, we can apply a similar argument, and thus, $p_{i,j}$ is not a stationary grid vertex.

In the next part, we prove (2). Without loss of generality, assume that $f(p)=0$.
Note that since $p$ is a local minimum, both eigenvalues $\lambda_1 \leq \lambda_2$ of the Hessian of $f$ at $p$ are positive.
Let $P_2$ denote the second order Taylor polynomial of $f$  centered at $p$.
Then $P_2$ is a quadratic form with eigenvalues $\frac{\lambda_1}{2} > 0$ and $\frac{\lambda_2}{2} > 0$, and the curve $P_2 = 1$ is an ellipse.
Now, since $f$ is $C^3$-class, there is some $\bar{\alpha} \in \Re$ such that for every $(x,y) \in D$, we have
\[
| f(x,y)-P_2(x,y) | < \frac{\bar{\alpha}}{\sqrt{2}} \left( |x|^3 + x^2 |y| + |x| y^2 + |y|^3 \right) = \frac{\bar{\alpha}}{\sqrt{2}} \left( |x| + |y| \right) \left( x^2 + y^2 \right) \leq \bar{\alpha} \left( x^2 + y^2 \right)^{3/2},
\]
which yields that for some suitable $\alpha \in \Re$, we have  $| f(q)-P_2(q) | \leq \alpha \left( P_2(q) \right)^{3/2}$ for every $q \in D$.
 Thus, for any $\varepsilon > 0$ there is a neighborhood $U$ of $p$ such that
\begin{itemize}
\item for every $q \in U$, we have $f(q) > 0$, and $| f(q)-P_2(q)| < \varepsilon P_2(q)$,
\item $f$ is convex in $U$.
\end{itemize}
Observe that the second condition holds for any convex neighborhood of $p$, where the Hessian of $f$ has only positive eigenvalues, and the existence of such a neighborhood follows from the fact that $f$ is $C^3$-class.
Now, since $P(q)$ is homogeneous, every point $q \in D$, with $f(q)=\alpha$, is contained between the ellipses $P_2(q) =  (1-\varepsilon) \alpha$ and $P_2(q) =  (1+\varepsilon) \alpha$. Note that if $\varepsilon$ is sufficiently small, for any value of $\alpha$ and any point $q$ of the level curve $f(x,y)=\alpha$, the angle between the two tangent lines of the ellipse $P_2(x,y) =  (1-\varepsilon) \alpha$, passing through $q$, is at least $\frac{\pi}{3}$.

Fix any `fine' equidistant partition $F_n$, and consider the level curves $f(x,y) = \alpha$, as $\alpha \geq 0$ increases. Let $\bar{p}$ be the first grid vertex that reaches the boundary of such a curve (note that according to our assumptions, there is a unique such grid vertex). Clearly, $f(\bar{p})$ is minimal among all the grid vertices in $U$.
Let
\begin{equation}\label{eq:r}
 r \geq \frac{\sqrt{a^2+b^2}}{\min \{ a,b\}} \cdot \frac{\lambda_2}{\lambda_1} \cdot \sqrt{\frac{1+\varepsilon}{1-\varepsilon}} .
\end{equation}
In the remaining part of the proof of (2), we show that there is no other grid vertex in $U$ which is minimal within its grid circle of radius $r$.

Assume, for contradiction, that the grid vertex $q$ is minimal within $C_r(q)$, and let $f(q)=\beta$. Then the level curve $f(x,y)=\beta$ already contains some grid vertex $q'$ in its interior. Note that the semi-axes of the ellipse $P_2(x,y)=t$ are of length $\sqrt{\frac{2t}{\lambda_i}}$, where $i=1,2$. Recall that the curve $f(x,y)=\beta$ is contained in the ellipse $P_2(x,y) = (1+\varepsilon) \beta$, and the diameter of the latter curve is $2\sqrt{\frac{2(1+\varepsilon)\beta}{\lambda_1}}$. Since, according to our assumption, $q'$ is contained in the interior of $P_2(x,y) = (1+\varepsilon) \beta$, and $f(q') < f(q)$, we obtain that
\begin{equation}\label{eq:whatweknow}
r \delta < 2\sqrt{\frac{2(1+\varepsilon)\beta}{\lambda_1}},
\end{equation}
where $\delta = \min \left\{ \frac{a}{n}, \frac{b}{n} \right\}$ denotes the minimal distance between any two grid vertices.

Let $w$ be the point of $P_2(x,y) = (1-\varepsilon) \beta$ closest to $q$. Let $\Delta = \frac{\sqrt{a^2+b^2}}{n} = \frac{\sqrt{a^2+b^2}}{\min \{ a,b\}} \delta$, and observe that any circle of diameter $\Delta$ contains a grid vertex. We show that the circle $C$ of diameter $\Delta$, touching the ellipse $P_2(x,y) = (1-\varepsilon) \beta$ at $w$ from inside, is contained in the ellipse. By Blaschke's Rolling Ball Theorem  \cite{Blaschke}, to do this it suffices to show that $\frac{\Delta}{2}$ is not greater than any radius of curvature of the ellipse.
It is a well-known fact that the radius of curvature at any point of an ellipse with semi-axes $M \geq m$ is at least $\frac{m^2}{M}$ and at most $\frac{M^2}{m}$.
Thus, a simple computation yields that what we need to show is
\begin{equation}\label{eq:whatweneed}
\Delta \leq 2 \frac{\sqrt{2(1-\varepsilon) \beta \lambda_1}}{\lambda_2}.
\end{equation}
To show (\ref{eq:whatweneed}), we can combine (\ref{eq:whatweknow}) with the definition of $r$ in (\ref{eq:r}).

Let $\bar{C}$ be the circle of radius $\Delta$ that touches the tangent lines of the ellipse $P_2(x,y) = (1-\varepsilon) \beta$ through $q$.
Since $f$ is convex in $U$, the level curve $f(x,y)=\beta$ is also convex, and thus, this circle is also contained inside the level curve $f(x,y)=\beta$. On the other hand, $\bar{C}$ as any other circle of diameter $\Delta$, contains a grid vertex $q''$. Then, our previous observation yields that $f(q'') < \beta = f(q)$. To finish the proof, we show that $\bar{C}$ is contained in the circle of radius $r \delta$, centered at $q$, which implies that $q''$ is contained in the grid circle of radius $r$, centered at $q$.

Assume, for contradiction, that it is not so. Let $\phi$ be the angle between the two tangent lines of the ellipse $P_2(x,y)=(1-\varepsilon) \beta$, through $q$. Since the angle between these two tangent lines is at least $\frac{\pi}{3}$, a simple computation yields that the distance of the center of $\bar{C}$ and $q$ is at most $\Delta$, and hence no point of $\bar{C}$ is farther from $q$ than $\frac{3}{2} \Delta = \frac{3\sqrt{a^2+b^2}}{2 \min \{ a,b\}} \delta \leq r \delta$, which finishes the proof of (2).

To prove (3), we can apply (2) for the function $-f$.

Finally, we prove (4). Let $f(p) = 0$. Then, in a neighborhood $U$ of $q$, the set $\{ f(q) = 0 \}$, $q \in U$ can be decomposed into the union of two $C^2$-class curves, crossing each other at $q$, and for any $\alpha \neq 0$, the set $\{ f(q) = \alpha \}$, $q \in U$ is the union of two disjoint, $C^2$-class curves.
Furthermore,  if $U$ is sufficiently small, there is some sufficiently small $\phi>0$ and $\varepsilon > 0$ such that for any $q \in U$
\begin{itemize}
\item there is a closed angular domain $A$ with apex $q$ and angle $\phi$ such that for any point $q' \in A$ with $0 < |q'-q| < \varepsilon$, we have $f(q) < f(q')$;
\item there is a closed angular domain $B$ with apex $q$ and angle $\phi$ such that for any point $q' \in B$ with $0 < |q'-q| < \varepsilon$, we have $f(q) > f(q')$.
\end{itemize}
Clearly, for a sufficiently large $r$ (chosen independently of $q$), any such closed angular domain in $U$ contains a vertex of $C_r(q)$, which yields the assertion.
\end{proof}

\begin{thm}\label{thm:main}
Le $f$ have $S$ local minima and $U$ local maxima. Then there is some $r$ such that for any sufficiently large $n$, exactly $S$ grid vertices of $F_n$ are minimal, and exactly $U$ grid vertices of $F_n$ are maximal within their grid circles of radius $r$.
\end{thm}

\begin{proof}
Fix some $r$ such that any stationary point $q$ of $f$ has some neighborhood that satisfies the corresponding conditions in (2), (3) or (4) of Theorem~\ref{thm:auxiliary}.
Observe that we can choose $\varepsilon_1,\varepsilon_2>0$ such that
\begin{itemize}
\item if $q$ is a stationary point, the assertion in (2), (3) or (4) Theorem~\ref{thm:auxiliary} holds in the $\varepsilon_1$-neighborhood $U_q$ of $q$;
\item if $q$ is not a stationary point, and its distance from any stationary point is at least $\varepsilon_1$, then (1) holds in the $\varepsilon_2$-neighborhood $U_q$ of $q$.
\end{itemize}

Now, let $n$ be large enough such that for any point $q \in D$, $C_r(q) \subset U(q)$, and for any stationary point, (2), (3) and (4) can be applied, and then, the theorem follows.
\end{proof}

\begin{rem}
We may apply Theorem~\ref{thm:main} for a parametrized convex surface $r=r(u,v)$, with $x,y$ as $(u,v)$, and $z=f(x,y)$ as the distance function $\| r(u,v)\|$.
\end{rem}

\subsection{Numerical example for evolving surface displaying the singular limit for local equilibria}\label{ss:3Dnum}

Here we give a 3D illustration for the phenomenon described in Theorem \ref{thm:caustic}. However, instead of regarding an evolution of the surface itself which would induce a simultaneous evolution for the caustics and the reference point we only treat a simpler case where both the surface and the caustics are constant and we move the reference point along a self-defined trajectory.

\begin{figure}
\includegraphics[width=0.95\textwidth]{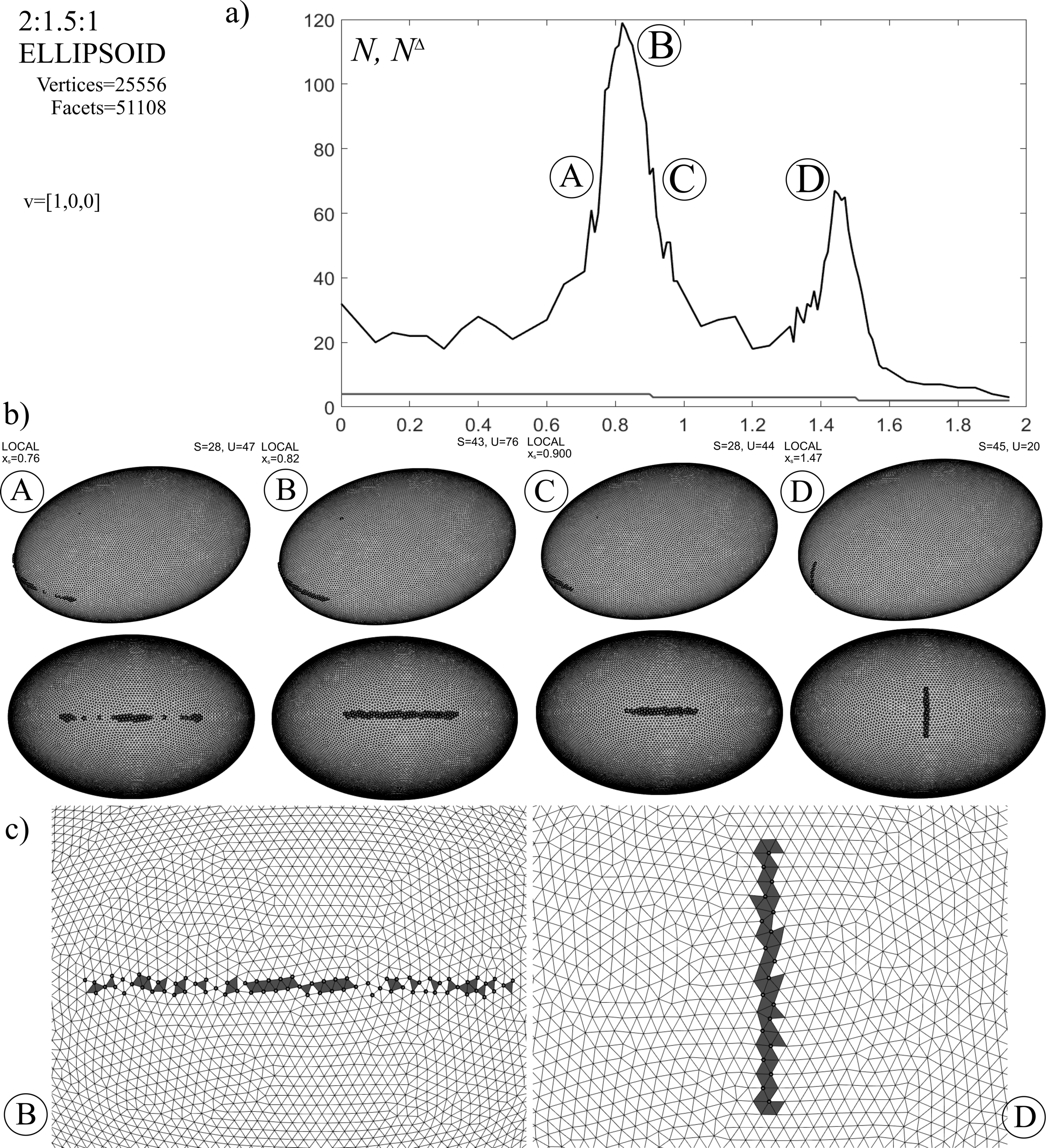}
\caption[]{The reference point is moved along the major axis of an ellipsoid. Panel a) shows the co-evolution of the number of global (red) and local (black) equilibria. The two peaks in the number of the local equilibria is due to the double intersection with the caustics. Panel b) is four snapshots about the emerging flocks on the surface from a general viewpoint (top) and a vantage point along the major axis (bottom). Panel c) shows enlarged versions of the B and D flocks from panel b).}
\label{fig:generic}
\end{figure}

The surface used for these computations is a triangulated ellipsoid with $a=2.0, b=1.5, c=1.0$. The approximately equidistant triangulation of the surface with $V=25556$ vertices and $F=51108$ facets was generated by DistMesh \cite{distmesh}. Offsetting the reference point in the direction $v=(1,0,0)$ the caustics is crossed twice, (Figure \ref{fig:generic}). Note, that $v$ is directed along the major axis, hence it produces reference points which are in the $(xy)$ and $(xz)$ planes, both are planes of symmetry of the object. Both intersections with the caustics produces peaks in the number of local equilibria, and due to symmetry these crossings belong to the same surface point, $p=(2,0,0)$. Observe that the spatial expansion of the local equilibria in any of the flocks reveals the line of curvature on the surface at $p$. As the reference point is moved from the center, the first crossing of the caustics takes place at the smaller principal curvature at $p$; it is straightforward that the flock unfold in the horizontal plane (cases A, B and C on the figure). Similarly, the second crossing of the caustics (associated with the higher principal curvature) is associated with a flock spread in the vertical plane (case D).

\begin{figure}
\includegraphics[width=0.95\textwidth]{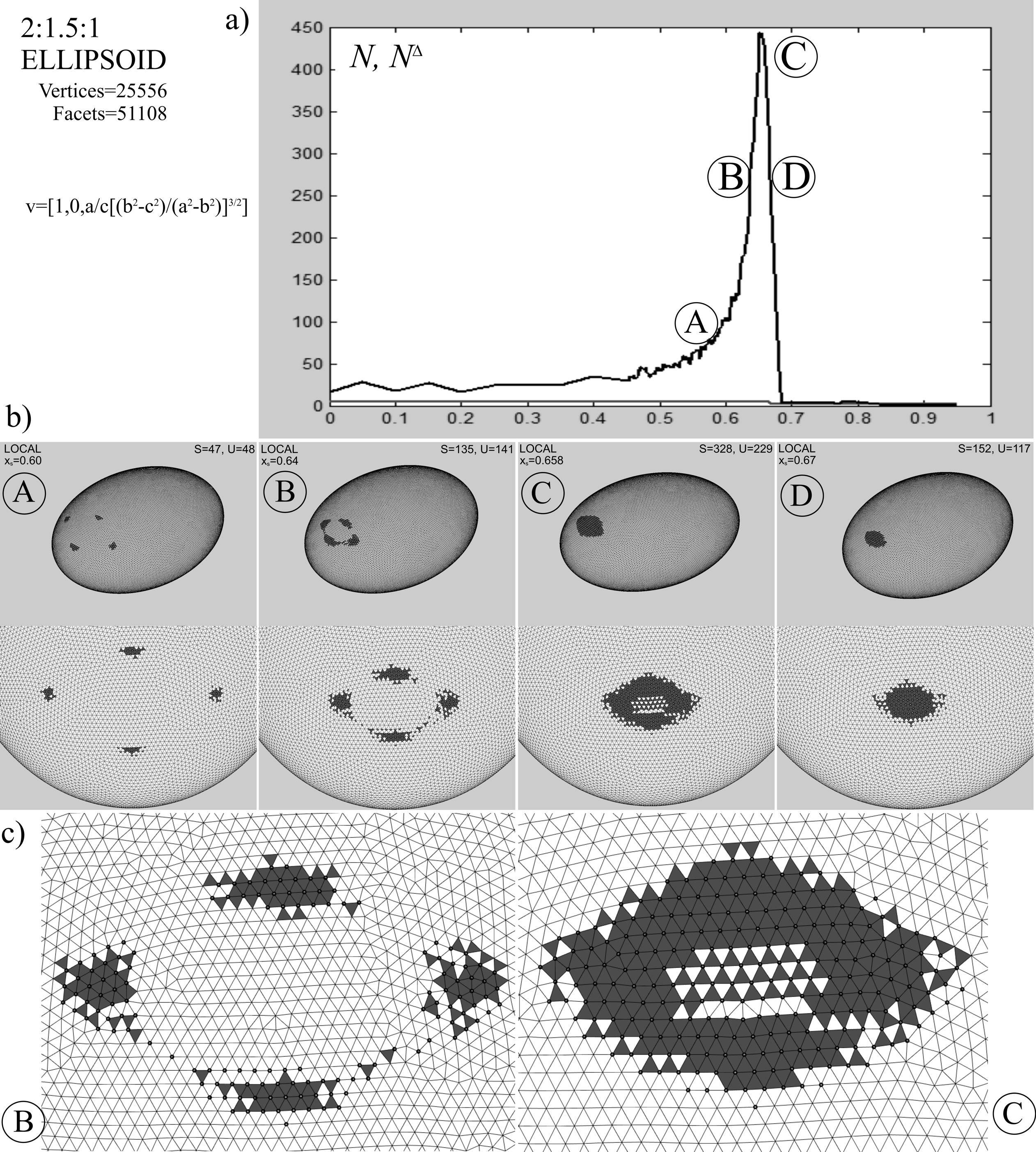}
\caption[]{The reference point is moved towards the point of the caustics that corresponds to one of the umbilical points of the surface. Panel a) shows the co-evolution of the number of global (red) and local (black) equilibria. The single peak in the number of the local equilibria is due to the single crossing with the caustics. Panel b) is four snapshots about the emerging flocks on the surface from a general viewpoint (top) and a vantage point along the major axis (bottom). Panel c) shows enlarged versions of the B and C flocks from panel b). }
\label{fig:umbilic}
\end{figure}

The umbilical point of the ellipsoid represents a special case: as the two principal curvatures are equal, any direction in the tangent plane is tangential to a line of curvature. Hence, we expect a spatially distributed flock. This phenomenon is illustrated in Figure \ref{fig:umbilic}. The displacement of the reference point takes place in the 
\begin{align}
\nonumber v=\left(\frac{1}{a}\frac{(a^2-b^2)^{3/2}}{(a^2-c^2)^{1/2}},0,\frac{1}{c}\frac{(b^2-c^2)^{3/2}}{(a^2-c^2)^{1/2}}\right)
\end{align}

\noindent direction, the distance of the caustics from the center is $\left\|v\right\|\cong 1.0477$. In accordance with Theorem 3, the number of local equilibria suddenly drops as the caustics is crossed, indicating an annihilation of global equilibria. Here two saddles, a stable and an unstable balancing points merge to form a single stable equilibrium.

\section{Discussion and conclusions}\label{Discussion}
 In this paper we constructed a theory connecting the number $N(t)$ of static equilibrium points on one-parameter
families of smooth curves to the evolution of the number $N^{\Delta}(t)$ of equilibrium points on their finely discretized approximations. First we show that if $r(t)$ is non-smooth then
the relationship between  $N^{\Delta}(t)$ and $N(t)$ may be rather different.

\subsection{Non-smooth evolution}

As pointed out in Theorem 4 in \cite{Matgeo}, smooth, generic bifurcations of equilibria may only occur if the \emph{spatial order} 
(i.e. the order of the highest spatial derivative) in the evolution equation is at least two. This is the case for curvature-driven evolution, however, there are other
evolution equations highly relevant for abrasion models which are of lower order. One of the most prominent examples is the Eikonal equation which may be written as
\begin{equation}\label{eikonal}
v=1,
\end{equation}
where $v$ denotes the speed in the direction of the inward surface normal. Equation (\ref{eikonal}) may be also written in polar coordinates, in two dimensions in the
PDE notation for the radial evolution of a curve $r(\varphi)$ as
\begin{equation}\label{PDE}
\frac{\partial r}{\partial t}=\frac{1}{r}\sqrt{r^2+\left(\frac{\partial r}{\partial \varphi}\right)^2}.
\end{equation}
As we can see, the Eikonal equation is of first spatial order.
Unlike curvature-driven flows, (\ref{PDE}) does not preserve the smoothness of the evolving manifold.  As long as $r(t)$ remains smooth (approximately until
$t=1700$ on Figure \ref{fig:eikonal}), $N(t)$ remains constant and
once $r(t)$ becomes non-smooth, $N(t)$ decreases monotonically \cite{Eikonal}. However, the coupling between  $N^{\Delta}(t)$ and $N(t)$
is in this case rather different: here the downward jumps of $N(t)$ are \emph{not} coupled to
resonance-like ($A$-type) events in the evolution of $N^{\Delta}(t)$, rather, both evolutions have a downward trend (see our example illustrated in Figure \ref{fig:eikonal}).

\begin{figure}
\includegraphics[width=0.90\textwidth]{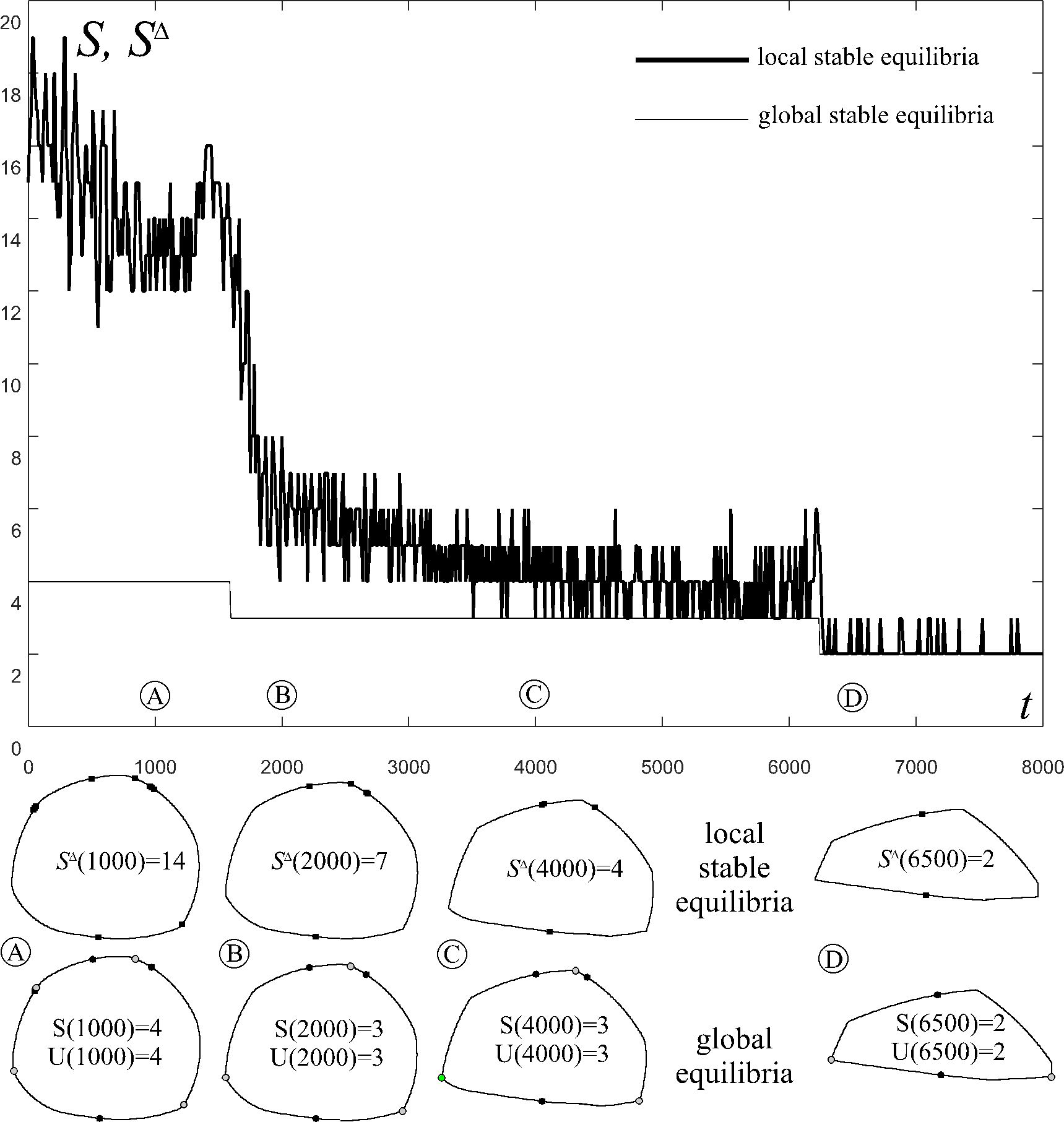}
\caption[]{ Co-evolution of local and global equilibria in 2D under the Eikonal equation (\ref{PDE}). An initially smooth curve with $4$ stable and $4$ unstable balance points is evolved
in manner such that the perimeter of the shape is kept unit. At $t\approx 1700$ vertices evolve and thus the curve ceases to be smooth. The simulation is carried out until $2$ stable and $2$ unstable (global) equilibria remain. Note stark contrast with Figure \ref{fig:coev} in the evolution of $S^{\Delta}(t)$}
\label{fig:eikonal}
\end{figure}

It is worth noting that if $r$ is a polytope then, for sufficiently fine mesh-size we may choose discretizations where edges and vertices
of $r^{\Delta}$ coincide with the edges and vertices of $r$. In this case, evidently, we have $N^{\Delta}(t) \equiv N(t)$ and we can observe 
a closely related scenario for $t>4000$ in Figure \ref{fig:eikonal}.  This observation illustrates that in this
case the tendency of the evolution of $N^{\Delta}(t)$ and  $N(t)$ is similar, in stark contrast with the smooth scenario discussed in the main body of the paper
(compare Figures \ref{fig:coev} and \ref{fig:eikonal}).

\subsection{Related other phenomena}

Our analysis shows that for smooth functions $r(t)$ the evolutions of $N(t)$ and $N^{\Delta}(t)$
are strongly coupled and the evolution in the discretized system can help to
forecast changes in the smooth system. In particular, \emph{downward} jumps in $N(t)$ are preceded by resonance-like
divergence in $N^{\Delta}(t)$. While the evolution of $N(t)$ is characteristic of the physical process, both in computer simulations and laboratory experiments $N^{\Delta}(t)$ is the observable quantity, so these results provide a tool to understand
the former by observing the latter.

The fact that a discretization may carry information relevant to predict the behavior of the underlying original system
has been observed before, although in quite a different context. In case of the Gibbs phenomenon 
one tries to reconstruct a signal with jump discontinuity by using
 partial sum of harmonics. However, no matter how many harmonics are included
in the partial sum, the original signal is recovered with a significant error because large oscillations occur
near the discontinuity.  As the frequency of the added harmonics increases, the overshoot does not die out, rather, it approaches a finite limit. By monitoring the oscillations due to the overshoot,
the discretized system can be used to forecast the jump in the original system. While the discretization happened in a function space (rather than in physical space), nevertheless, the Gibbs phenomenon is still reminiscent of the phenomena
described in our paper. The appearance of 'tygers' in the discretized Burgers and Euler equations \cite{tyger}, \cite{villani}
is analogous to the Gibbs phenomenon, however, here we regard the discretization of solutions to evolution equations. Similarly to the Gibbs phenomenon, here also the sudden jump (shockwave) in the solution  is preceded by large oscillations in the Fourier approximation and thus a critical event in the continuous system is reflected by resonance-type behavior in the corresponding discretization. 

Both previous examples referred to Fourier-type discretizations. It is also known that spatial discretizations (closer to the topic of our paper) may yield "parasitic" solutions not present in the continuous system. Most often parasitic solutions are regarded as a mere numerical embarrassment \cite{doedel}, nevertheless, the example of
\emph{ghost solutions} in elasticity \cite{ghosts} shows that, similarly to the current problem, they could also contribute to the understanding of the underlying continuous system.

\section{Acknowledgement}
The authors are most indebted to Phil Holmes for reading the initial manuscript and giving
essential advice on several aspects. We also thank L\'aszl\'o Sz\'ekelyhidi for drawing our attention to the analogy with tygers.  
This research has been supported by the NKFIH grant K 119245.

\end{document}